\documentclass[a4paper,11pt]{amsart}
\usepackage[mathscr]{eucal}
\usepackage{graphicx}
\usepackage{float}
\restylefloat{figure}

\newtheorem{theorem}{Theorem}[section]
\newtheorem{lemma}[theorem]{Lemma}

\newtheorem{assumption}{Assumption}

\newtheorem{definition}[theorem]{Definition}

\newtheorem{remark}{Remark}

\date{\today}
\begin{document}

\title[Adaptive timestepping, stability, positivity]{Adaptive timestepping for pathwise stability and positivity of strongly discretised nonlinear stochastic differential equations}
\author[C.~Kelly]{C\'onall Kelly}
\address{Department of Mathematics, The
University of the West Indies, Kingston, Jamaica. }
\email{conall.kelly@uwimona.edu.jm}
\author[A.~Rodkina]{Alexandra Rodkina}
\address{Department of Mathematics, The
University of the West Indies, Kingston, Jamaica. }
\email{alexandra.rodkina@uwimona.edu.jm}
\author[E. Rapoo]{Eeva Maria Rapoo}\address{Science Campus, Department of Statistics, 
The University of South Africa, Johannesburg, South Africa}
\email{rapooe@unisa.ac.za}

\footnote{The work was initiated when the second author was visiting The University of South Africa, Johannesburg, South Africa.}

\begin{abstract}
We consider the use of adaptive timestepping to allow a strong explicit Euler-Maruyama discretisation to reproduce dynamical properties of a class of nonlinear stochastic differential equations with a unique equilibrium solution and non-negative, non-globally Lipschitz coefficients. Solutions of such equations may display a tendency towards explosive growth, countered by a sufficiently intense and nonlinear diffusion.

We construct an adaptive timestepping strategy which closely reproduces the a.s. asymptotic stability and instability of the equilibrium, and which can ensure the positivity of solutions with arbitrarily high probability. Our analysis adapts the derivation of a discrete form of the It\^o formula from Appleby et al (2009) in order to deal with the lack of independence of the Wiener increments introduced by the adaptivity of the mesh. We also use results on the convergence of certain martingales and semi-martingales which influence the construction of our adaptive timestepping scheme in a way proposed by Liu \& Mao (2017). 



\end{abstract}

\maketitle

\section{Introduction}
Consider the scalar stochastic differential equation (SDE) of It\^o type
\begin{eqnarray}
\label{eq:maincont1}
dX(t)&=&X(t)f(X(t)) dt+X(t)g(X(t))dW(t),\quad t\geq 0\\
X(t)&=&\varsigma\geq 0,\nonumber
\end{eqnarray}
where $(W(t))_{t\geq 0}$ is a one-dimensional Wiener process; let $(\mathcal{F}_t)_{t\geq 0}$ be the natural filtration of $W$. The drift and diffusion coefficients satisfy:
\begin{assumption}
\label{as:fg}
Let $f, g:\mathbb R\to [0, \infty)$ be non-negative functions such that $g(u)\neq 0$  for $u\neq 0$.
\end{assumption}

In this article, we use an adaptive timestepping strategy to reproduce qualitative properties of solutions of \eqref{eq:maincont1} in an explicit strong Euler-Maruyama discretisation given by
\begin{equation}\label{eq:gendiscr}
X_{n+1}=X_n\left(1+ h_nf(X_n)+\sqrt{h_n}g(X_n)\left[W(t_{n+1}-W_{t_n})\right]\right),\quad n\in\mathbb{N},
\end{equation}
with $X_0=X(0)=\varsigma$.  Each $h_n$ is one of a sequence of random timesteps, generated as a function of $X_n$, and we set $\{t_n:=\sum_{i=1}^{n}h_i\}_{n\in\mathbb{N}}$. 

Our first goal is to design a strategy that allows discretisations of the form \eqref{eq:gendiscr} to closely reproduce of the a.s. stability and instability of the unique equilibrium solution $X(t)\equiv 0$ of \eqref{eq:maincont1}.  The strategy will be required to capture the stabilising effect of the diffusion as it counters the tendency towards explosive growth due to the positive drift. Since we will use martingale and semimartingale convergence results in our analysis, we must adopt elements of the approach developed by Liu \& Mao~\cite{LM} in order to ensure that those results are applicable.  

Our second goal is to investigate the effect of our adaptive timestepping strategy on the probability of positivity of solutions of \eqref{eq:gendiscr}. Unique solutions of \eqref{eq:maincont1} when $\varsigma>0$ are necessarily positive, though a highly nonlinear diffusion coefficient makes it likely that trajectories of a fixed-step discretisation will overshoot the equilibrium and become negative.  Adaptive timestepping was succesfully used in \cite{AKR2010} to preserve positivity with high probability in equations with either a dominant nonlinear and strongly zero-reverting drift coefficient, or a dominant and highly variable nonlinear diffusion coefficient. That article was a follow up to \cite{AGKR}, and our analytic technique is adapted from both.

An analysis of the ability of explicit numerical methods with adaptive timesteps to reproduce the dynamics of solutions of \eqref{eq:maincont1} is important because explicit Euler methods of the form \eqref{eq:gendiscr} with constant stepsize $h_n\equiv h$ are known (see \cite{HJK2011}) to fail to converge strongly to solutions of \eqref{eq:maincont1} if either $f$ or $g$ grows superlinearly, as is the case for \eqref{eq:maincont}. Fixed-step taming methods were introduced first in \cite{HJK2012} to provide an alternative class of strongly convergent explicit methods for such equations, but may not provide an optimal reproduction of qualitative behaviour: see \cite{TZ2013,KL2017}. It was recently shown (see \cite{FG2016,KL2017}) that, for equations with one-sided Lipschitz drift and globally Lipschitz diffusion coefficients, adaptive timestepping strategies can be used to ensure strong convergence of solutions of the explicit Euler method with variable stepsizes, and therefore their effect on the dynamics of solutions is of interest: see \cite{FG2017}.  

Let us now consider a minimal set of additional constraints to place upon $f$ and $g$. Suppose first that $f$ and $g$ are locally Lipschitz continuous and that Assumption \ref{as:fg} holds. Then there exists a unique, continuous $\mathcal F_t$-measurable process $X$ (see \cite{M}, \cite{O}) satisfying  \eqref{eq:maincont1} on the interval $[0, \tau_e^\varsigma)$, where $\tau_e^\varsigma=\inf\{t>0: |X(t, \varsigma)|\notin [0, \infty)\}$.
Define the first hitting time of zero to be $\upsilon_e^\varsigma=\inf\{t>0: |X(t, \varsigma)|=0.\}$.
It was proved in \cite{AMR1} that $\upsilon_e^\varsigma=\tau_e^\varsigma=\infty$, and therefore unique positive solutions exist on all of $\mathbb{R}^+$, if 
\begin{equation}
\label{cond:stab}
\sup_{u\neq 0}\frac{2f(u)}{g^2(u)}=\beta<1.
\end{equation} 

Condition \eqref{cond:stab} is close to being sharp. \eqref{eq:maincont1} has an equilibrium solution $X(t)\equiv 0$, and (see \cite{AMR1}) if 
\begin{equation}
\label{cond:instab1}
\lim_{u\to 0}\frac{2f(u)}{g^2(u)}>1,
\end{equation}  
then this equilibrium is a.s. unstable: for all $\varsigma>0$,
\[
\mathbb{P}\left[\lim_{t\to\infty} X(t)=0\right]=0.
\]
Alternatively, if 
\begin{equation}
\label{cond:stab1}
\lim_{u\to 0}\frac{2f(u)}{g^2(u)}<1,
\end{equation}  
then for all $\varsigma>0$
\[
\mathbb{P}\left[\lim_{t\to\infty} X(t)=0\right]>0.
\]

Conditions \eqref{cond:stab} and \eqref{cond:stab1} require the diffusion coefficient $g$ to have a stabilising effect. For example, consider the scalar stochastic differential equation with positive polynomial coefficients
\begin{equation}
\label{eq:maincont}
dX(t)=X(t)\left(X^\nu(t) dt+\sigma X^{\nu/2}(t)dW_t\right),\quad t\geq 0, \quad X(0)=\varsigma\geq 0,
\end{equation}
where $\nu\in(0,\infty)$. In this case $f(u)=u^\nu$ and $g(u)=\sigma u^{\nu/2}$, and therefore \eqref{cond:stab} is satisfied with $\lim_{u\to 0}\frac{2f(u)}{g^2(u)}=2/\sigma^2<1$ when $\sigma^2>2$. So if the intensity of the stochastic perturbation is sufficiently large, unique positive solutions exist on $[0,\infty)$ and  converge to zero with positive probability. If $\sigma=0$, then \eqref{eq:maincont} becomes the ordinary differential equation
\[
x'(t)=[x(t)]^{1+\nu},\quad t\geq 0,\quad x(0)=\varsigma> 0,
\]
solutions of which exist only on the interval $[0,\tau_e^\varsigma)$ if
\begin{equation}\label{eq:tau}
\tau_e^\varsigma:=\int_{\varsigma}^{\infty}\frac{1}{u^{1+\nu}}du<\infty.
\end{equation}
In this case the existence of a unique global solution (which remains positive a.s.) and stability of the zero equilibrium with nonzero probability both require the presence of a sufficiently intense stochastic perturbation, due to the action of the positive drift. Note finally that the qualitative analysis of solutions of discrete-time equations in this article does not require $f$ and $g$ to be locally Lipschitz continuous. 

The layout of the article is as follows. In Section \ref{sec:mart}, we review basic ideas from the theory of martingales and present two useful convergence theorems. In Section \ref{sec:mot} we consider a discrete model of \eqref{eq:maincont1} with an i.i.d. innovation process, and show how an adaptive timestepping strategy that enforces a constant bound on the drift and diffusion responses allows us to invoke a.s. stability and instability results from \cite{ABR1}. However, this approach, which requires the use of a discrete form of the It\^o formula, fails when terms of the innovation process are not mutually independent. This is the case for the strong approximation given by \eqref{eq:gendiscr}, since the increments of $W$ are now taken over a random interval which depends on past values of $W$. Therefore, in Section \ref{sec:strong} we adjust our adaptive timestepping strategy to ensure that it generates a sequence of stopping times with respect to the natural filtration of $W$. This will allow us to apply martingale and semi-martingale convergence results from Section \ref{sec:mart} in the proof of our main results, which are presented in Section \ref{sec:main}. Numerical illustrations are given in Section \ref{sec:num}, and proofs of our main results are given in Section \ref{sec:proofs}.

\section{Martingales and convergence theory}\label{sec:mart}
In what follows, suppose that $(\Omega, {\mathcal{F}}, \{{\mathcal{G}}_n\}_{n \in \mathbf N},
{\mathbb P})$ is a complete filtered probability space. We recall the following definitions
\begin{definition}
\begin{enumerate}
\item A stochastic sequence $\{M_n\}_{n \in \mathbb N}$ is said to be a {\it
  $\mathcal{G}_n$-martingale}, if ${\mathbb E}|M_n|<\infty$ and $\mathbb E\left[M_n |
\mathcal{G}_{n-1}\right]=M_{n-1}$ for all $n\in\mathbb N$ {\it a.s.}
\item A stochastic sequence $\{\mu_n\}_{n \in \mathbb N}$ is said to be an {\it
$\mathcal{G}_n$-martingale-difference}, if $\mathbb E |\mu_n|<\infty$ and
$\mathbb E\left[\mu_n | \mathcal{G}_{n-1}\right]=0$ {\it a.s.} for all
$n\in\mathbb N$.
\end{enumerate}
\end{definition}
The following construction may be found in \cite{ABR1,KPR} and will be key to the proof of one of our main results:
\begin{lemma}\label{lem:prodMart}
Let $\{Y_i\}_{i\in\mathbb{N}}$ be a sequence of non-negative random variables defined on $(\Omega,\mathcal{F},\mathbb{P})$ and adapted to the filtration $\{\mathcal{G}_n\}_{n\in\mathbb{N}}$, where each $Y_i$ satisfies the following
\begin{enumerate}
\item[i)] $\mathbb{E}[Y_i]<\infty$;
\item[ii)] $\mathbb{E}[Y_i|\mathcal{G}_{i-1}]=1$. 
\end{enumerate}
Then the sequence $\{M_n\}_{n\in\mathbb{N}}$ given by
\[
M_n=\prod_{i=1}^{n}Y_i,\quad n\in\mathbb{N},
\]
is a $\mathcal{G}_n$-martingale.
\end{lemma}
\begin{proof}
See proof of Lemma 2.2 in \cite{KPR}.
\end{proof}
We now present two convergence results required for the analysis in this article. The first is a classical result on the convergence of non-negative martingales, which may be found (for example) in Shiryaev~\cite{Shiryaev96}:
\begin{lemma}\label{lem:nnM}
If $\{M_n\}_{n\in\mathbb{N}}$ is a non-negative $\mathcal{G}_{n}$-martingale, then $\lim_{n\to\infty}M_n$ exists with probability one.
\end{lemma}
The second result provides for the convergence of random sequences that admit to a particular kind of estimation from above, and was proved in \cite{AMaR}:
\begin{lemma}\label{lem:nonegdif}
  Let $\{Z_n\}_{n\in \mathbf N}$ be a non-negative  $\mathcal{F}_n$-measurable
  process, ${\mathbb E}|Z_n|<\infty$ for all $n\in \mathbb N$, and
  \begin{displaymath}
    Z_{n+1}\le Z_{n}+u_n-v_n+\mu _{n+1}, \quad n = 0, 1, 2, \dots,
  \end{displaymath}
  where $\{\mu_n\}_{n\in \mathbf N}$ is a ${\mathcal{G}}_n$-martingale-difference,
  $\{u_n\}_{n\in \mathbb N}$, $\{v_n\}_{n\in \mathbb N}$ are nonnegative $\mathcal{G}_n$-measurable
  processes and ${\mathbb E}|u_n|, {\mathbb E}|v_n|<\infty$ for all $n\in \mathbb N$.

  Then $$\left\{\omega: \sum_{n=1}^{\infty} u_n<\infty\right\}\subseteq \left\{\omega:
    \sum_{n=1}^{\infty} v_n<\infty\right\}\bigcap\{Z_n\to\}.$$
$\{Z_n\to\}=\{\omega\in\Omega\,:\,Z_n(\omega)\to\}$ denotes the set of all $\omega\in\Omega$ for which $\lim_{n\to+\infty}Z_n(\omega)$ exists and is finite.
\end{lemma}

\section{An adaptive timestepping strategy}\label{sec:adapt}

\subsection{Motivation: strategy for a discrete time model}\label{sec:mot}
We will start by investigating the behaviour of the stochastic difference equation
\begin{eqnarray}\label{eq:gendiscr2}
X_{n+1}&=&X_n\left(1+ h_nf(X_n)+\sqrt{h_n}g(X_n)\chi_{n+1}\right),\quad n\in\mathbb{N},\\
X_0&=&\varsigma>0.\nonumber
\end{eqnarray}
Each $\chi_n$ satisfies the following assumption:
\begin{assumption}
  \label{as:noise}
  Suppose that
 \begin{enumerate}
  \item $\chi_n$ are independent $\mathcal{F}_n$-measurable random variables satisfying
    \begin{equation*}
      \mathbb E\chi_n = 0, \qquad \mathbb E\chi_n^2=1, \qquad
      \mathbb E|\chi_n|^3 \mbox{ are uniformly bounded;}
    \end{equation*}
  \item the probability density function $p_n$ of each $\chi_n$ exists and satisfies
    \begin{equation*}
      x^3 p_n(x) \to 0 \quad \mbox{as } |x|\to\infty \quad
      \mbox{uniformly in } n.
    \end{equation*}
  \end{enumerate}
\end{assumption}
 Eq. \eqref{eq:gendiscr2} is not the strong Euler-Maruyama discretisation of \eqref{eq:maincont1}, since the innovation process $\{\chi_n\}_{n\in\mathbb{N}}$ cannot represent a sample of the increments of $W$: see Section \ref{sec:strong} for more detail. 

Consider this discrete form of the It\^o formula, which was presented and proved in \cite{ABR1}: 
\begin{lemma}
\label{thm:Idiscrto}
Consider $\phi:\mathbb R\to \mathbb R$ such that there exists $\delta>0$ and
  $\bar \phi:\mathbb R\to \mathbb R $ satisfying
  \begin{enumerate}
  \item $\bar \phi  \equiv \phi$ on $U_\delta = [1-\delta,1+\delta]$,
  \item $\bar \phi \in C^3(\mathbb R)$ and $\left|\bar \phi'''(x)\right|\leq M$ for some $M$
    and all $x\in\mathbb R$,
  \item $\int_{\mathbb R} \left|\phi - \bar \phi\right| dx <
    \infty$.
  \end{enumerate}
Let $\chi$ be measurable with respect to the filtration $\{\mathcal{G}_n\}_{n\in\mathbb{N}}$ and satisfy Assumption~\ref{as:noise}. Let $f_n$ and $g_n$ be $\mathcal{G}_n$-measurable uniformly bounded random variables.
  Then there exists $h_0>0$ such that for all $h\le h_0$  and $n\in \mathbb N$, 
  \begin{multline*}
    \label{eq:Itoformula}
\mathbb E \left[\left. \phi\left(1+f_nh + g_n\sqrt{h}\chi_{n+1}\right) \right| \mathcal{G}_{n+1} \right ]
    \\= \phi(1) + \phi'(1) f_n h + \frac{\phi''(1)}{2} g_n^2 h
    + O(h^{3/2})[f_n  + g_n^2],
  \end{multline*}
where $|O(u)|\le K|u|$  for some $K>0$ and all $u\in \mathbb R$.
\end{lemma}
Lemma \ref{thm:Idiscrto} was used in \cite{ABR1} to provide conditions for a.s. asymptotic stability (using $\phi(u)=u^\alpha$) and instability (using $\phi(u)=u^{-\alpha}$) for the equilibrium solution of \eqref{eq:gendiscr} when $f$ and $g$ are bounded, $\alpha\in (0, 1)$ in both cases. In this article $f$ and $g$ are not bounded. However, we can exploit the analysis in \cite{ABR1} by constructing an adaptive timestepping strategy that ensures boundedness of the drift and diffusion response at each step.

Define, for some $\bar h>0$ and for each $n\in \mathbb N$
\begin{equation}
\label{def:hn0}
h_n:=\frac {\bar h}{1+f(x_n)+g^2(x_n)}.
\end{equation}
Let, also, for each $u\in  \mathbb R$,
\begin{equation}
\label{def:PhiG}
\Phi(u):=\frac {f(u)}{1+f(u)+g^2(u)}, \quad \Gamma(u):=\frac {g(u)}{\sqrt{1+f(u)+g^2(u)}}.
\end{equation}
Using notation \eqref {def:PhiG} we can write equation \eqref {eq:gendiscr} as
\begin{equation}
\label{eq:discrPhG}
X_{n+1}=x_n\biggl(1+ \bar h \Phi(X_n)+\sqrt{\bar h}\Gamma (X_n)\chi_{n+1}\biggr), \quad x_0=\varsigma>0, \,\, n\in\mathbf N.
\end{equation}
We can apply \cite[Theorems 6 \& 8]{ABR1} to \eqref{eq:discrPhG} and arrive at the following result.
\begin{theorem}
\label{thm:stabinstabnottrue}
Let $\{X_n\}_{n\in\mathbb{N}}$ be a solution of \eqref{eq:gendiscr} where each $h_n$ is defined by \eqref{def:hn0}, where $f$ and $g$ satisfy Assumption \ref{as:fg}, and where terms of the sequence $\{\chi_n\}_{n\in\mathbb{N}}$ satisfy Assumption \ref{as:noise}. 
\begin{enumerate}
\item [(i)] If \eqref{cond:stab} holds then there exists $h_0>0$ such that for all $\bar h\le h_0$ the equilibrium solution of \eqref{eq:discrPhG} is globally a.s. asymptotically stable: for each $\varsigma>0$
\[
\lim_{n\to \infty} X_n=0,\quad a.s.
\]
\item  [(ii)]  If 
\begin{equation}
\label{cond:instab}
\liminf_{u\neq 0}\frac{2f(u)}{g^2(u)}>1,
\end{equation}  
then there exists $h_0>0$ such that for all $\bar h\le h_0$ the equilibrium solution of \eqref{eq:discrPhG} is a.s. unstable: for each $\varsigma>0$
\[
\mathbb P\left[\lim_{n\to \infty} X_n=0\right]=0.
\]
\end{enumerate}
\end{theorem}
\begin{proof}
Solutions of \eqref {eq:discrPhG} are also solutions of \eqref{eq:gendiscr2}, and vice versa. Note that, for each $u\in  \mathbb R$ 
\[
0\le \Phi(u)\le 1, \quad  0\le \Gamma (u)\le 1,
\]
and $\Phi$ and $\Gamma$ satisfy all conditions of Assumption \ref{as:fg}. Note also that Condition \eqref{cond:stab} implies that 
$
\sup_{u\neq 0}\frac{2\Phi (u)}{\Gamma^2(u)}<1,
$
and condition \eqref{cond:instab} implies that 
$
\liminf_{u\neq 0}\frac{2\Phi (u)}{\Gamma^2(u)}>1.
$
\end{proof}
Therefore we see that, for the discrete-time model \eqref{eq:gendiscr2}, an adaptive $h_n$ may be constructed at each step which bounds the response of the discrete-time drift and diffusion coefficients (even if the drift and diffusion coefficients of the corresponding stochastic differential equation are unbounded) and allows us to apply existing results on a.s. stability and instability from \cite{ABR1}.

\subsection{Strategy for the strong Euler-Maruyama discretisation}
\label{sec:strong}
The strong Euler-Maruyama discretisation of \eqref{eq:maincont1} is given by
\begin{equation}\label{eq:Euler}
X_{n+1}=X_n\left(1+h_nf(X_n)+g(X_n)\triangle W_{n+1}\right),\quad n\in\mathbb{N},
\end{equation}
where $\triangle W_{n+1}:=W(t_{n+1})-W(t_n)$. Unfortunately, we cannot now directly apply the discrete It\^o formula given by Lemma \ref{thm:Idiscrto} in Section \ref{sec:mot}, since terms of the sequence $\{\Delta W_{n}\}_{n\in\mathbb{N}}$ are not independent: the length of the increment $h_{n}=t_{n+1}-t_n$ depends on $X_n$, and therefore on the Wiener path up to time $t_n$. Therefore Assumption \ref{as:noise} is not satisfied, and we must modify the approach that leads to Lemma \ref{thm:Idiscrto}. Since the innovation process directly samples trajectories of $W$, it is now necessary to ensure that each $t_n$ is an $\mathcal{F}_t$-stopping time, where $(\mathcal{F})_{t\geq 0}$ is the natural filtration of $W$, in order that the appropriate semi-martingale convergence theory may be applied. The importance of this issue was raised for the first time in the context of Euler-Maruyama methods with random variable stepsizes in Liu \& Mao~\cite{LM}.

To this end, we must modify the sequence defined by \eqref{def:hn0} to
\begin{equation}\label{def:hnTrue}
h_n=\frac{\bar h}{1+\lfloor f(X_n)\rfloor+\lfloor g^2(X_n)\rfloor}
\end{equation}
where $\lfloor a\rfloor$ denotes the integer part of $a\in [0, \infty)$, so that $\lfloor a\rfloor+1>a$ and $\bar h>0$ is a small convergence parameter. 

\begin{definition}
Suppose that each member of the sequence $\{t_n\}_{n\in\mathbb{N}}$ is an $\mathcal{F}_t$-stopping time: i.e. $\{t_n\leq t\}\in\mathcal{F}_t$ for all $t\geq 0$ where $(\mathcal{F}_t)_{t\geq 0}$ is the natural filtration of $W$. We may then define a discrete time filtration $\{\mathcal{F}_{t_n}\}_{n\in\mathbb{N}}$ by 
\[
\mathcal{F}_{t_n}=\left\{A\in\mathcal{F}\,:\,A\cap\{t_n\leq t\}\in\mathcal{F}_t\right\},\quad n\in\mathbb{N}.
\]
\end{definition}

\begin{lemma}
Consider the sequence $\left\{t_n:=\sum_{i=1}^n h_n\right\}$, where each $h_n$ is defined by \eqref{def:hnTrue}. Then each $t_n$ is an $\mathcal{F}_t$-stopping time.
\end{lemma}
\begin{proof}
The proof follows the procedure described in Step 1 of Theorem 3.1 in  Liu \& Mao~\cite{LM}.
\end{proof}
As in \cite{LM}, we note that the modified form of the sequence $h_n$ in \eqref{def:hnTrue} is motivated by the need for each timestep to be rational, which is automatically the case when the method is implemented on a finite state machine. 

\begin{remark}\label{rem:condMoments}
Since $t_n$ and $t_{n-1}$ are $\mathcal{F}_t$-stopping times and $\mathcal F_{t_{n-1}}$-measurable, $\Delta W_{n+1}$ is $\mathcal F_{t_{n}}$-conditionally normally distributed, with conditional distribution
\begin{equation}
\label{est:DeltaWn}
\Psi_{n+1}(t)=\frac 1{\sqrt{2\pi (t_{n+1}-t_{n})}}\int_{-\infty}^t e^{-\frac {x^2}{2(t_{n+1}-t_{n})}}ds.
\end{equation}
Moreover, there exists $K>0$ such that the first three conditional moments of $\triangle W_{n+1}$ obey
 \begin{eqnarray*}
\mathbb E\left[\Delta W_{n+1}\bigr| \mathcal F_{t_{n}}  \right]&=&0,\quad a.s.;\\
\mathbb E\left[|\Delta W_{n+1}|^2\bigr| \mathcal F_{t_{n}}  \right]&=&t_{n+1}-t_{n},\quad a.s.;\\
\mathbb E\left[|\Delta W_{n+1}|^3\bigr| \mathcal F_{t_{n}}  \right]&\leq& K|t_{n+1}-t_{n}|^{3/2},\quad a.s.\\
\end{eqnarray*}
\end{remark}

\begin{remark}
The notion of an admissible adaptive timestepping scheme was introduced in Definition 2.3 of  \cite{KL2017} as a way of ensuring the strong convergence of an adaptive explicit Euler method for SDEs with one-sided Lipschitz drift and globally Lipschitz diffusion coefficients. Note that the conditions placed upon $f$ and $g$ in the present article are significantly weaker, and in general a one-sided Lipschitz condition does not hold. So we cannot assume or expect strong convergence of the method, rather we are only concerned with the preservation of dynamics in the discretisation. Nor do we impose maximum and minimum stepsizes $h_{\text{max}}$ and $h_{\text{min}}$ for the theoretical analysis here. Such bounds are essential to the convergence analysis in \cite{KL2017}, but are not required for an analysis of discrete dynamics.
\end{remark}

\section{Main results}\label{sec:main}
All proofs are deferred to Section \ref{sec:proofs}.
\subsection{A.s. stability and instability}
In this section we present sufficient conditions on solutions of \eqref{eq:Euler} for solutions to demonstrate a.s. stability and instability. 
\begin{theorem}\label{thm:ASstab}
Let $\{X_n\}_{n\in\mathbb{N}}$ be a solution of \eqref{eq:Euler} with initial value $X_0=\varsigma>0$, and the sequence $\{h_n\}_{n\in\mathbb{N}}$ satisfy \eqref{def:hnTrue}. Suppose that $f$ and $g$ satisfy the conditions of Assumption \ref{as:fg} and there exists $\beta<1$ such that
\begin{equation}
\label{cond:stabbeta}
\sup_{u\neq 0}\frac{2f(u)}{g^2(u)}=\beta.
\end{equation} 
Then there exists $h_0>0$ such that for all $\bar h<h_0$,
\[
\lim_{n\to\infty}X_n=0,\quad a.s.
\]
\end{theorem}

\begin{theorem}\label{thm:ASinstab}
Let $\{X_n\}_{n\in\mathbb{N}}$ be a solution of \eqref{eq:Euler} with initial value $X_0=\varsigma>0$, and the sequence $\{h_n\}_{n\in\mathbb{N}}$ satisfies \eqref{def:hnTrue}. Suppose that $f$ and $g$ satisfy the conditions of Assumption \ref{as:fg} and there exists $\gamma>1$ such that
\begin{equation}
\label{cond:instabgamma}
\liminf_{u\to 0}\frac {2f(u)}{g^2(u)}=\gamma.
\end{equation}
Then there exists $h_0>0$ such that for all $\bar h<h_0$,
\[
\mathbb{P}\left[\lim_{n\to\infty}X_n=0\right]=0.
\]
\end{theorem}

\subsection{Positivity}
In this section we suppose that the number of timesteps $N\in\mathbb{N}$ is fixed, and provide conditions for a lower bound on the probability that the corresponding solution values of \eqref{eq:Euler} with adaptive timesteps satisfying \eqref{def:hnTrue} remain positive. If the estimate is pre-determined then its value influences a constraint on the range of possible values of the parameter $\bar h$.
\begin{theorem}\label{thm:pos}
\label{thm:pos}
Let $\{X_n\}_{n\in\mathbb{N}}$ be a solution of \eqref{eq:Euler} with initial value $X_0=\varsigma>0$, and the sequence $\{h_n\}_{n\in\mathbb{N}}$ satisfies \eqref{def:hnTrue}. Suppose that $f$ and $g$ satisfy the conditions of Assumption \ref{as:fg}. Then, for each $\varepsilon\in(0,1)$ there exists  $\bar h(\varepsilon)>0$ such that, for all $\bar h\in (0, \bar h(\varepsilon))$
\[
\mathbb P[X_N>0, X_{N-1}>0, \dots, X_{0}>0]>1-\varepsilon.
\]
\end{theorem}
The value of $\bar h(\varepsilon)$ may itself be estimated: see Remark \ref{rem:pos}.

\section{A numerical example}
\label{sec:num}
The equation \eqref{eq:maincont} with $\nu=2$ and $\varsigma=1$, may be written
\begin{equation}
dX(t)=X^3(t)dt+\sigma X^2(t)dW_t,\quad t\geq 0,\quad X(0)=1.\label{eq:polynum}
\end{equation}
We will use the adaptive timestepping rule defined by \eqref{def:hnTrue} to approximate trajectories of \eqref{eq:polynum} via the explicit Euler-Maruyama method:
\begin{equation}\label{eq:polynumEM}
X_{n+1}=X_n+h_nX_n^3+\sigma X_n^2 \triangle W_{n+1},\quad n\geq 0,\quad X_0=1.
\end{equation}
Note that $2f(u)/g^2(u)=2/\sigma^2<(>)1$ if $\sigma>(<)\sqrt{2}$. 

In Figure \ref{fig:plotsUnstab} we illustrate the case where $\sigma<\sqrt{2}$, displaying a trajectory and the corresponding time series of stepsizes taken by the method for $\sigma=0$ (first row), and $\sigma=1$ (third row) with $\bar h=1$. Instability and an apparent finite-time explosion is observed in each case, and each trajectory has only positive values, as expected. In the second row, we have reduced the maximum timestep for the $\sigma=0$ trajectory to $\bar h=0.1$, and we observe that the apparent explosion time now occurs closer to $\tau_e^\varsigma=0.5$, as computed for the underlying SDE via \eqref{eq:tau}.

The average stepsize for the $\sigma=0$ trajectory with $\bar h=1$ is 0.0017706, with $\bar h=0.1$ is $1.0832\times 10^{-4}$, and for the $\sigma=1$ trajectory it is $6.3593\times 10^{-6}$. We note that the timesteps quickly approach zero as the trajectory grows and so in the second column of Figure \ref{fig:plotsUnstab} we only show the magnitude of the first $100$ timesteps for clarity of presentation.

In Figure \ref{fig:plotsStab} we illustrate the case where $\sigma>\sqrt{2}$, displaying a trajectory and the corresponding time series of stepsizes taken by the method for $\sigma=2$ (first row), and $\sigma=3$ (second row) with $\bar h=1$. Apparent asymptotic stability is observed in each case, though negative values are observed when $\sigma=3$. By reducing the maximum $\bar h$ to $0.1$ we are able to preserve the positivity of the path (third row). Note that, as each trajectory approaches the equilibrium at zero, the adaptive steps settle at a value close to the maximum possible value $\bar h$. 

The average stepsize for the $\sigma=2$ trajectory is 0.96442, for the $\sigma=3$ trajectory with $\bar h=1$ it is $0.38223$, and for the $\sigma=3$ trajectory with $\bar h=0.1$ it is $0.097674$. We show the magnitude of all computed timesteps in the second column of Figure \ref{fig:plotsStab}.

In all cases, note that the timesteps are indexed against step number, whereas the trajectory is plotted in time. Since the mesh is non-uniform and different for each trajectory, care must be taken when making comparisons.

\begin{figure}
\begin{center}
$\begin{array}{@{\hspace{-0.3in}}c@{\hspace{-0.3in}}c}
\mbox{\bf\small Trajectory} & \mbox{\bf\small Stepsizes}\\
\scalebox{0.36}{\includegraphics{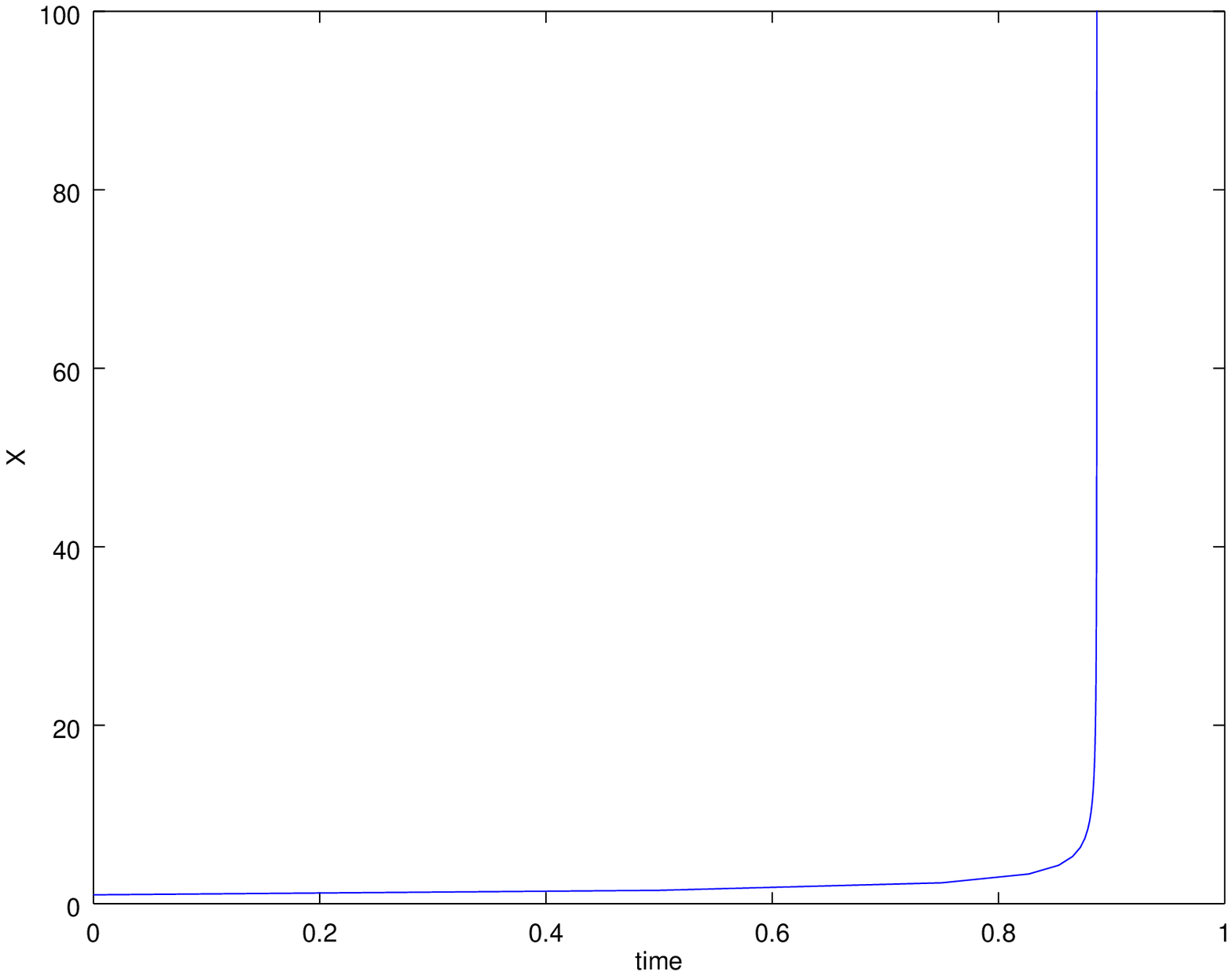}} & \scalebox{0.36}{\includegraphics{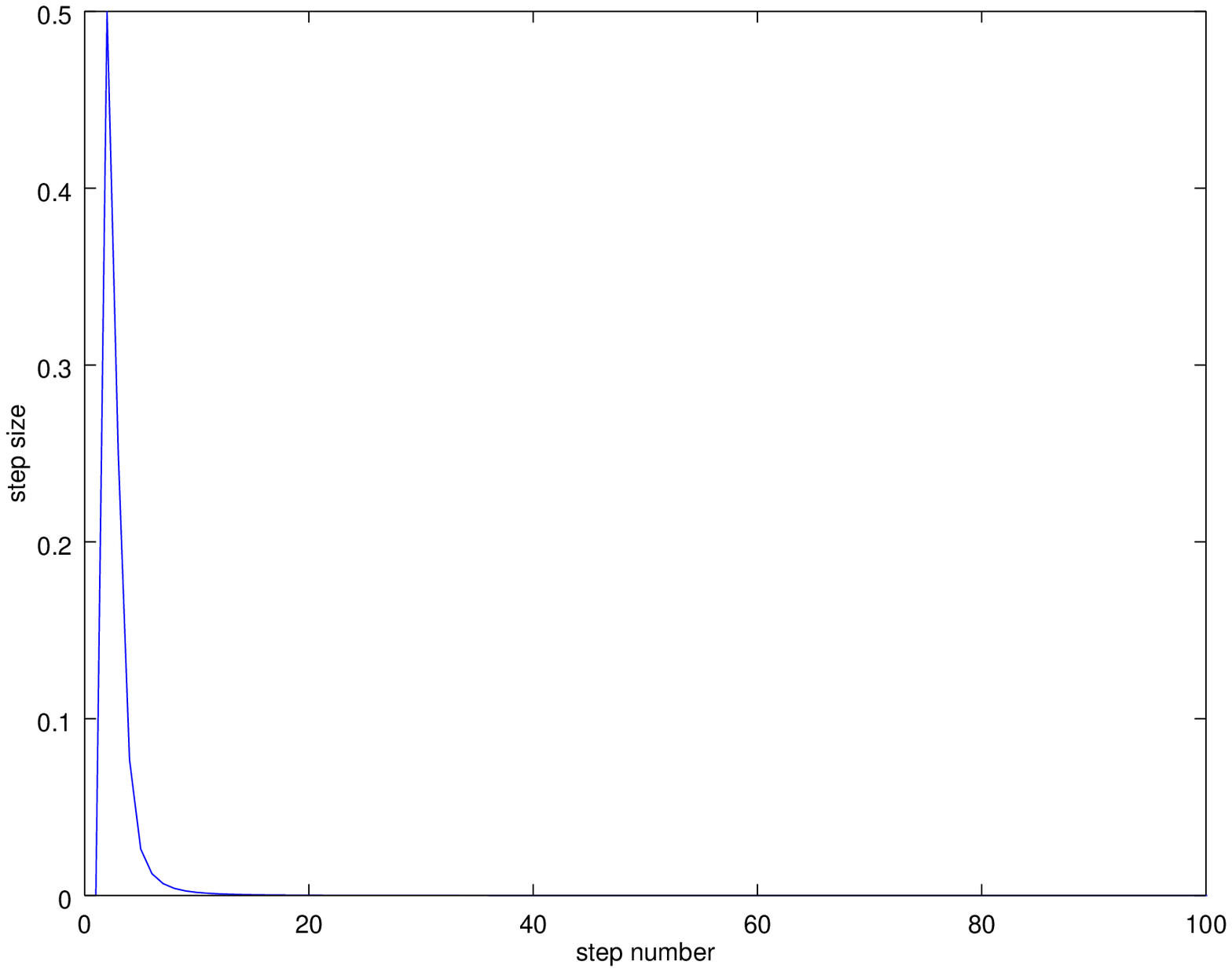}}\\ \mbox{\bf\small $\sigma=0$, $\bar h=1$} & \mbox{\bf\small $\sigma=0$, $\bar h=1$}\\
\scalebox{0.36}{\includegraphics{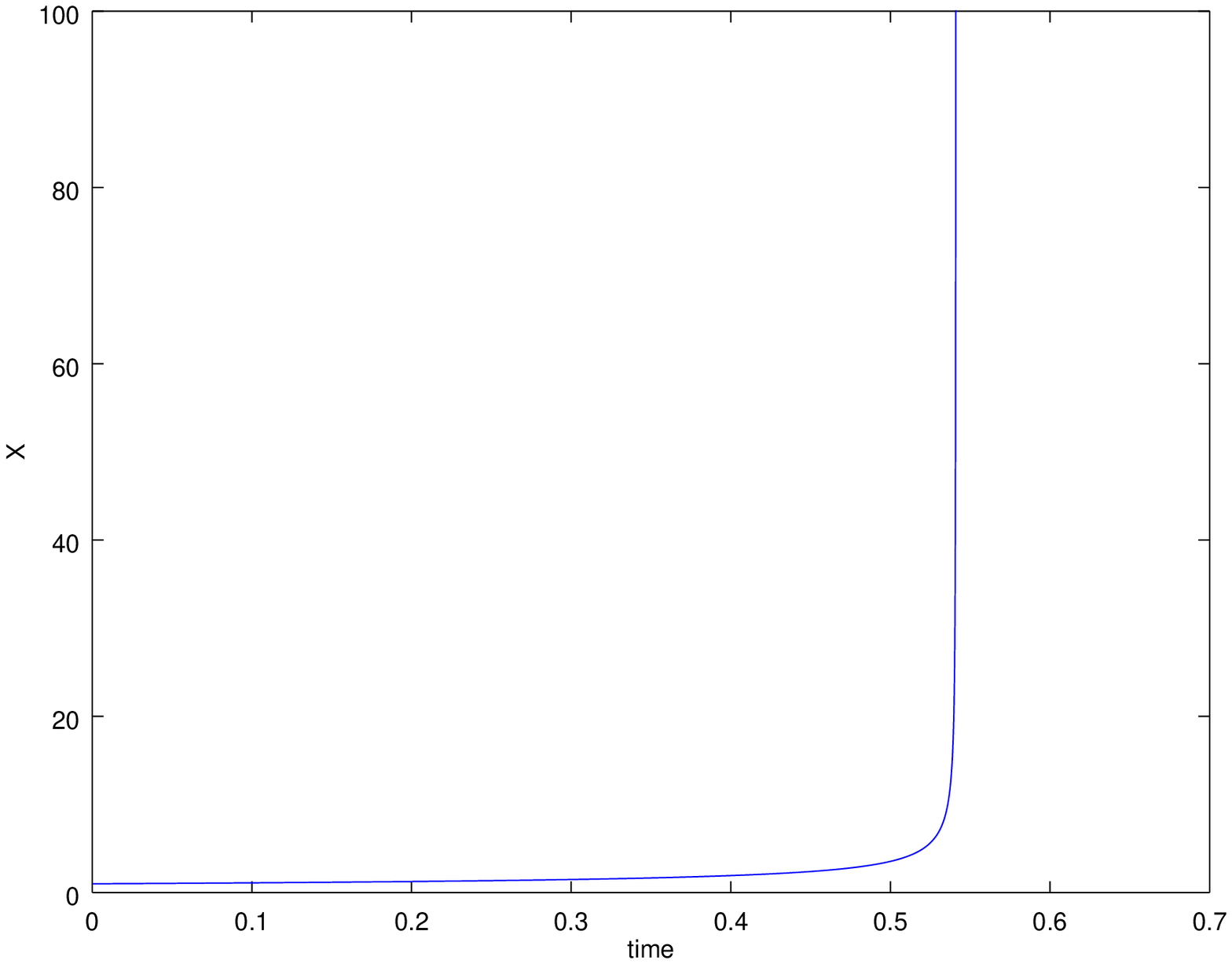}} & \scalebox{0.36}{\includegraphics{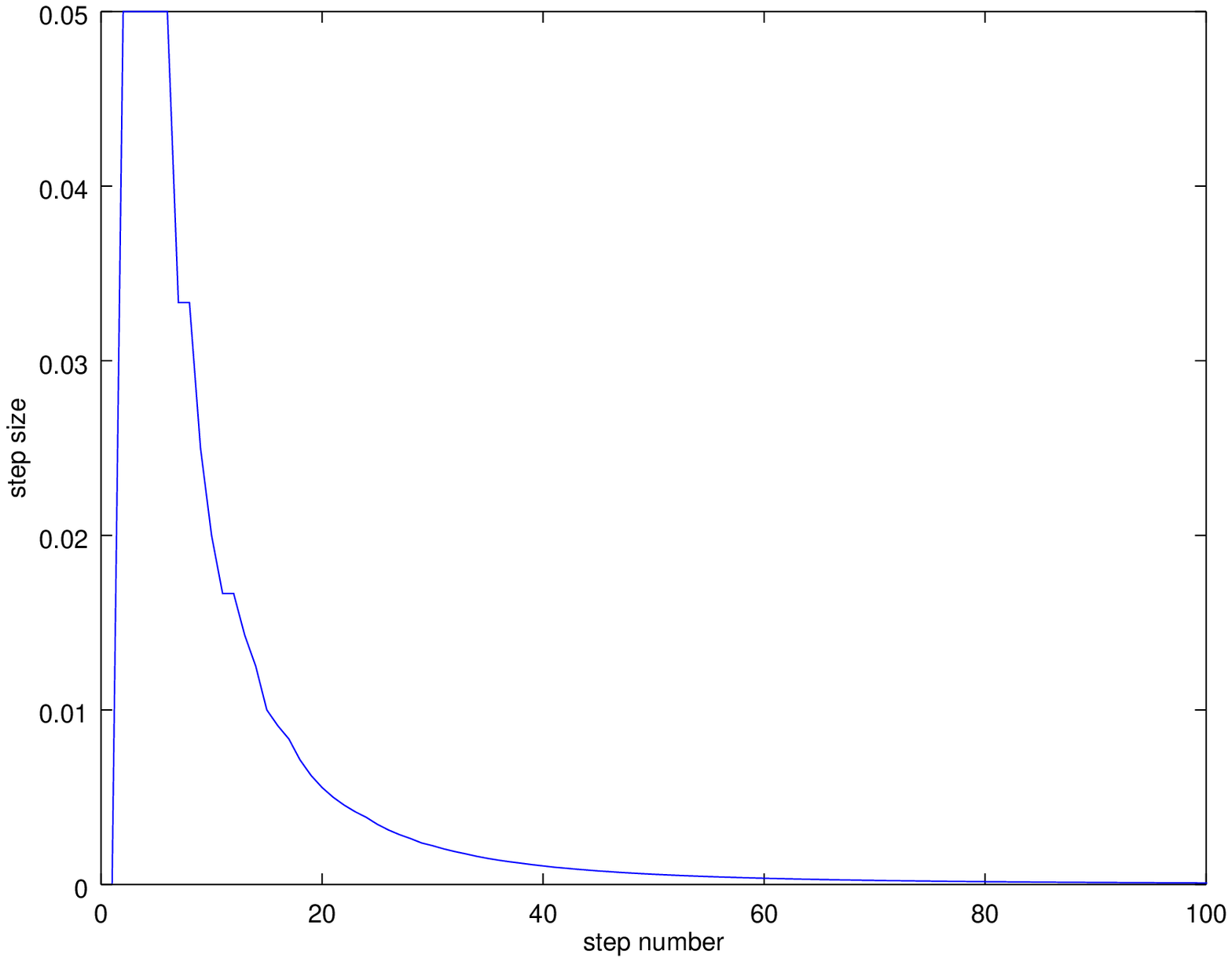}}\\ \mbox{\bf\small $\sigma=0$, $\bar h=0.1$} & \mbox{\bf\small $\sigma=0$, $\bar h=0.1$}\\
\scalebox{0.36}{\includegraphics{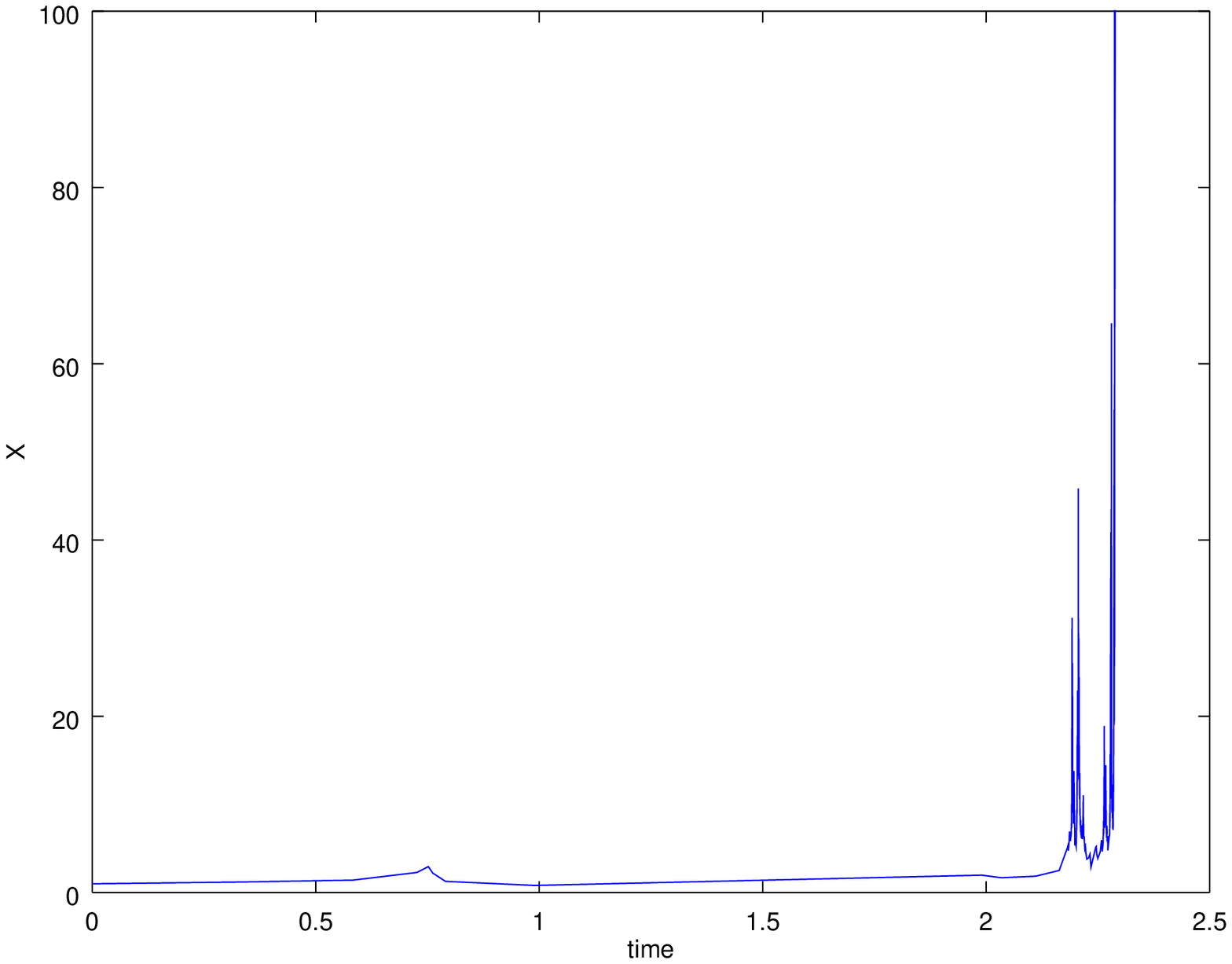}} & \scalebox{0.36}{\includegraphics{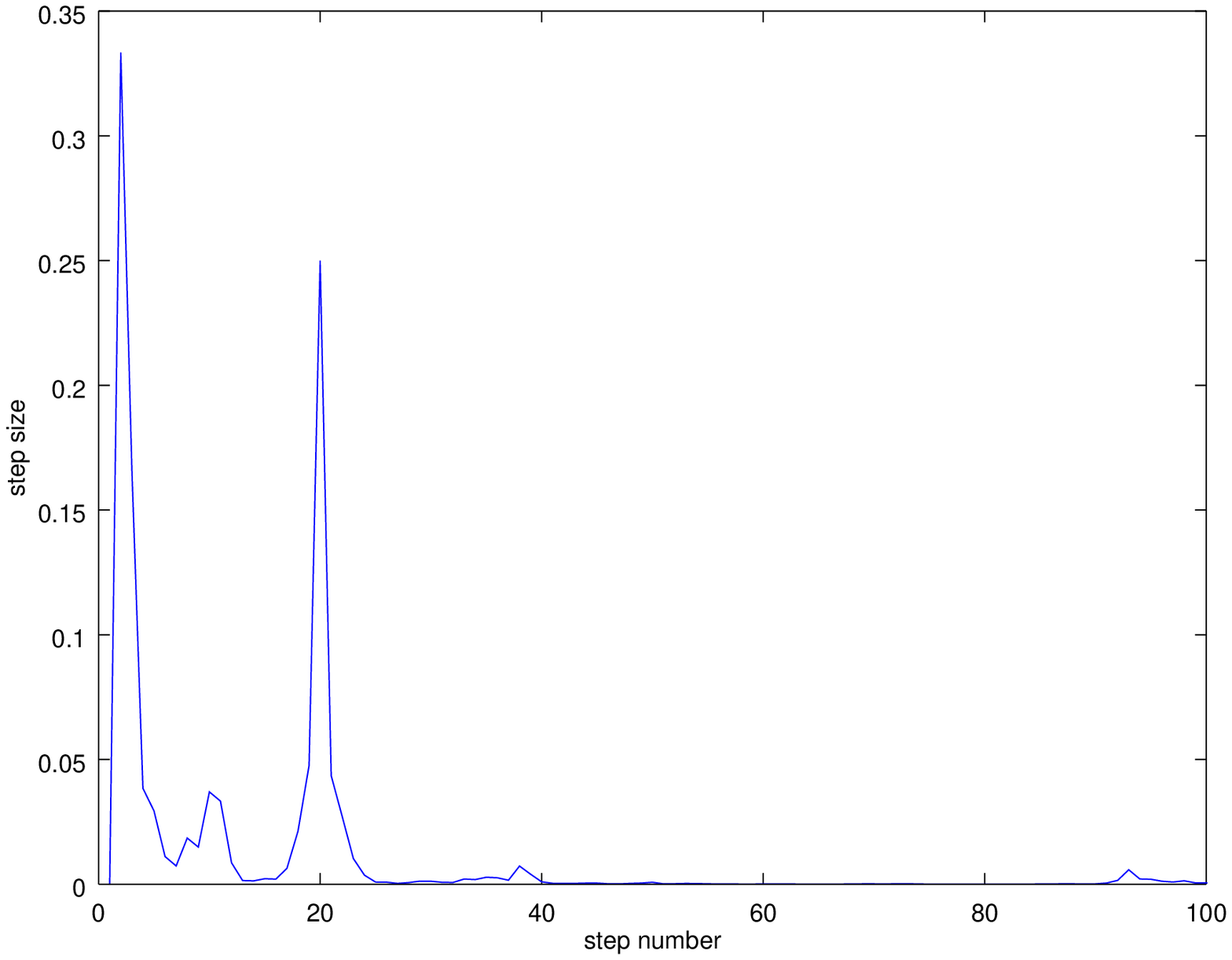}}\\ \mbox{\bf\small $\sigma=1$, $\bar h=1$} & \mbox{\bf\small $\sigma=1$, $\bar h=1$}\\\end{array}$
\end{center}
\caption{Trajectories and adaptive timestep sizes for \eqref{eq:polynumEM} with $\sigma=0$ (first row) and $\sigma=1$ (second row).}\label{fig:plotsUnstab}
\end{figure}
\begin{figure}
\begin{center}
$\begin{array}{@{\hspace{-0.3in}}c@{\hspace{-0.3in}}c}
\mbox{\bf\small Trajectory} & \mbox{\bf\small Stepsizes}\\
\scalebox{0.36}{\includegraphics{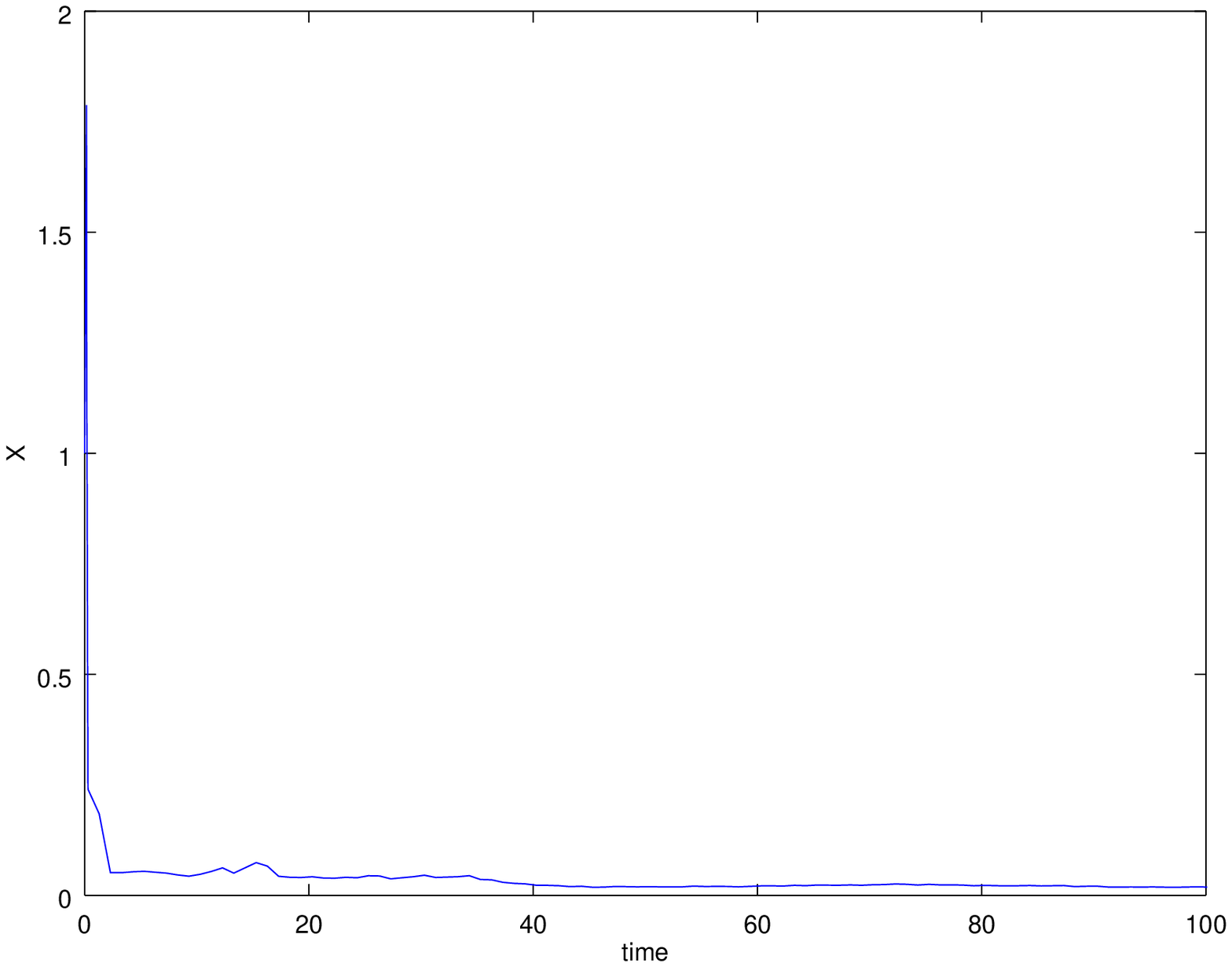}} & \scalebox{0.36}{\includegraphics{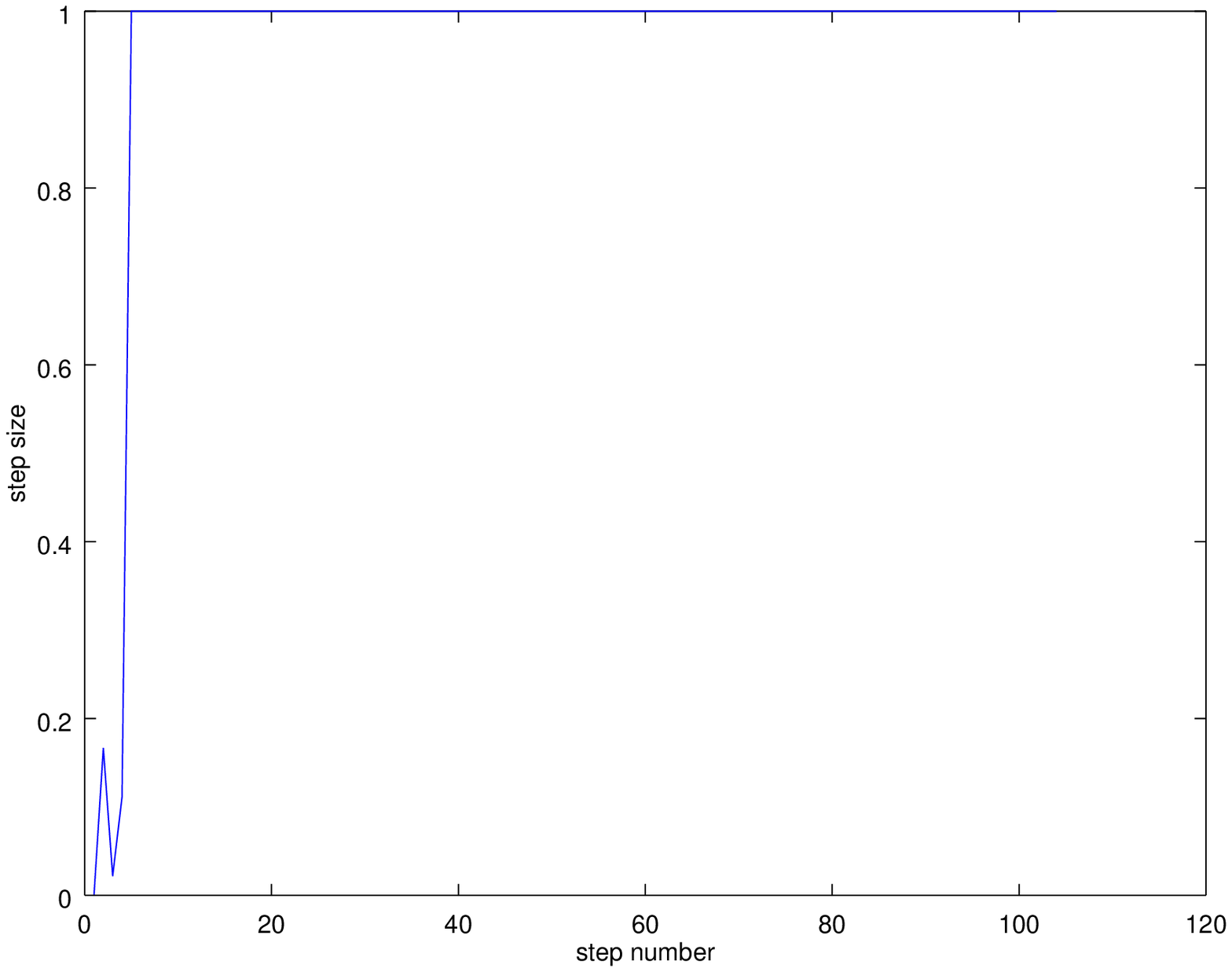}}\\ \mbox{\bf\small $\sigma=2$, $\bar h=1$} & \mbox{\bf\small $\sigma=2$, $\bar h=1$}\\
\scalebox{0.36}{\includegraphics{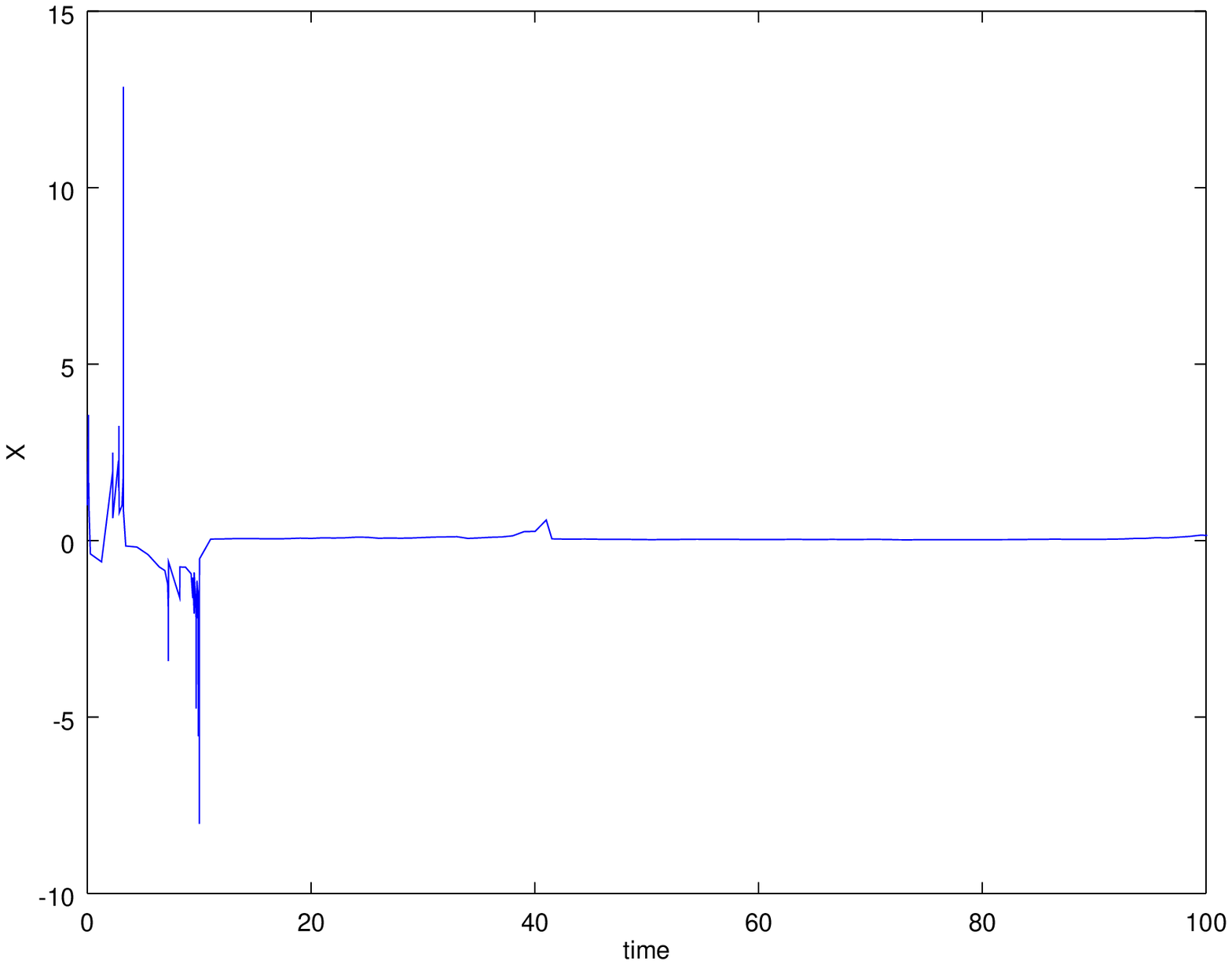}} & \scalebox{0.36}{\includegraphics{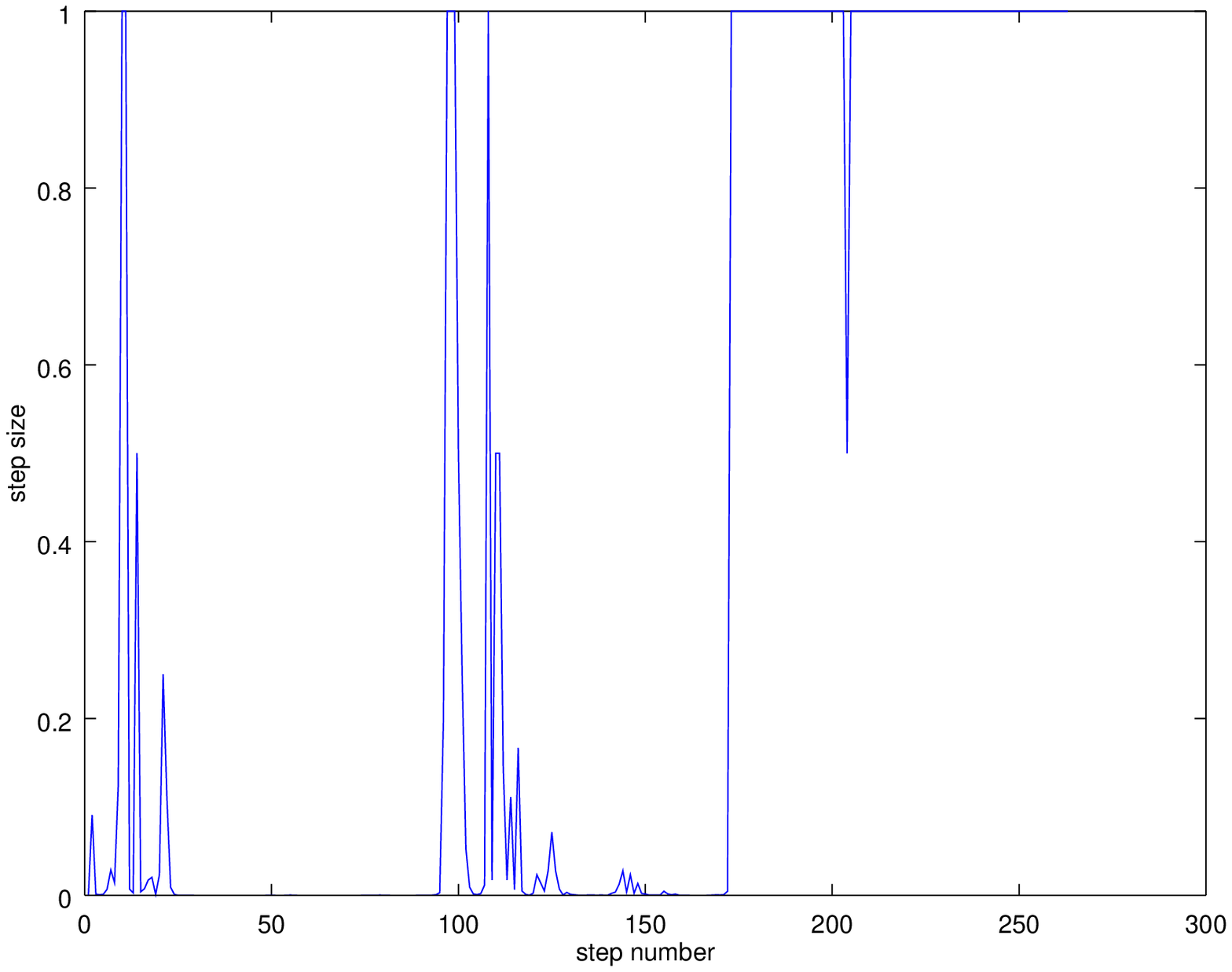}}\\ \mbox{\bf\small $\sigma=3$, $\bar h=1$} & \mbox{\bf\small $\sigma=3$, $\bar h=1$}\\
\scalebox{0.36}{\includegraphics{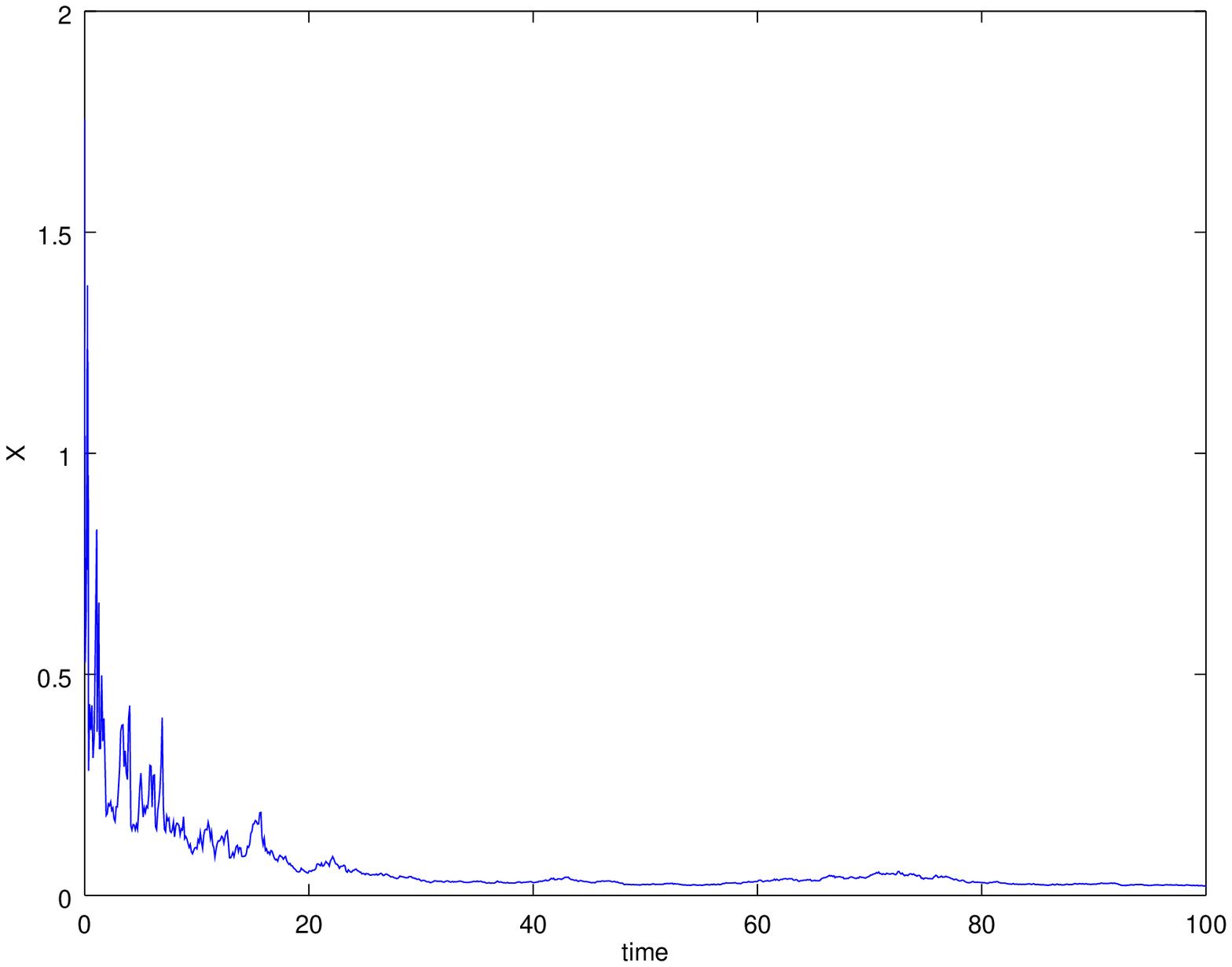}} & \scalebox{0.36}{\includegraphics{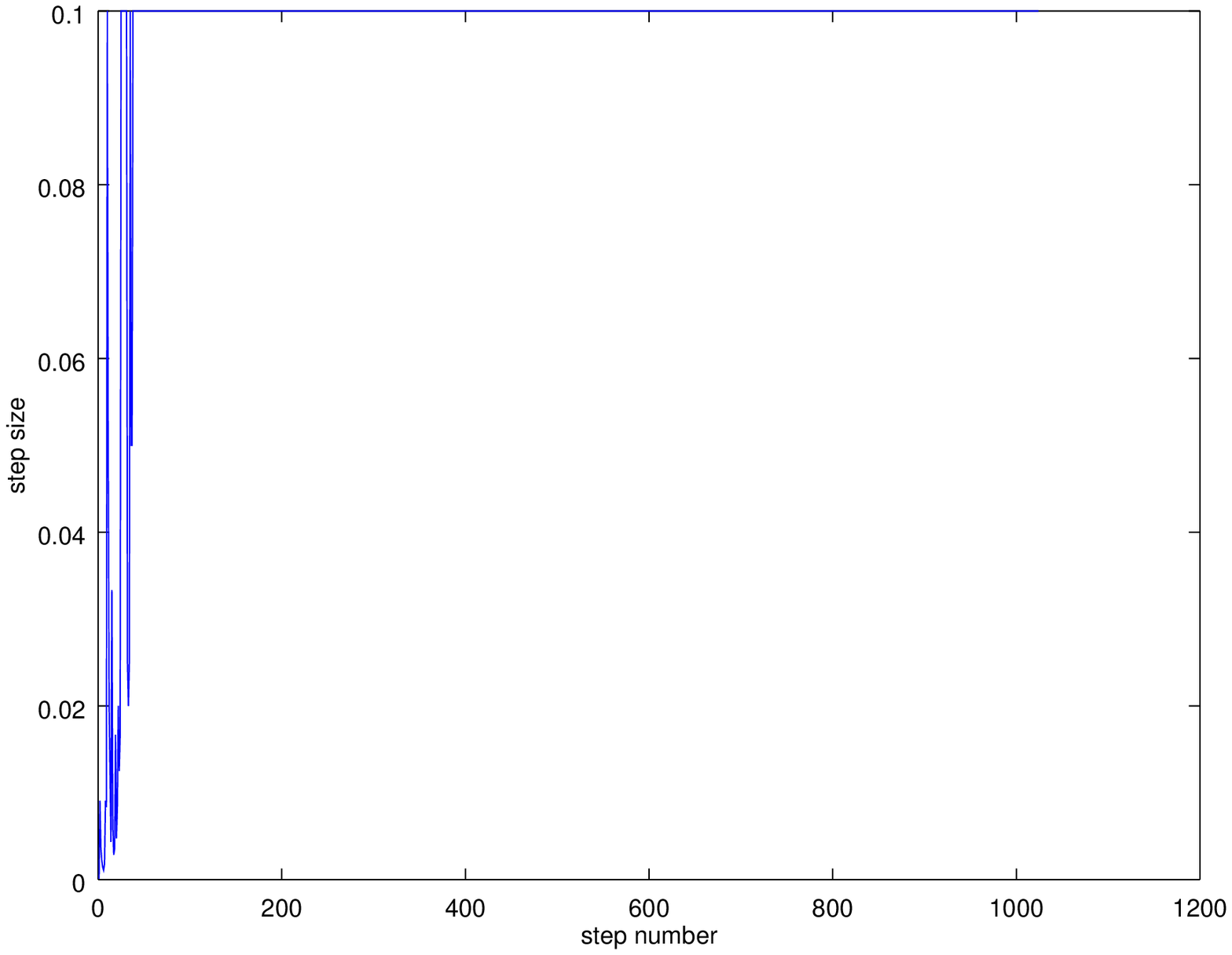}}\\ \mbox{\bf\small $\sigma=3$, $\bar h=0.1$} & \mbox{\bf\small $\sigma=3$, $\bar h=0.1$}\\

\end{array}$
\end{center}
\caption{Trajectories and adaptive timestep sizes for \eqref{eq:polynumEM} with $\sigma=2$ (first row) and $\sigma=3$ (second and third rows).}\label{fig:plotsStab}
\end{figure}

\section{Proofs}
\label{sec:proofs}
\subsection{Technical construction}
We begin with the following useful lemma which may be found in, for example, \cite{RD}.
\begin{lemma}
\label{lem:3mom}
Let $Y$ be a random variable with probability distribution function $\Phi$ and bounded third moment: $\mathbb E |Y|^3\le K$. Then
\[
\Phi(u)|u|^3\le K, \quad \text{for all}\,\, u<0.
\]
\end{lemma}
\begin{lemma}
Let $\{X_n\}_{n\in\mathbb{N}}$ be a solution of \eqref{eq:Euler}, where $f$ and $g$ satisfy Assumption \ref{as:fg} and $h_n$ is defined as in  \eqref{def:hnTrue}. Then the following estimates hold on all trajectories:
\begin{equation}
\label {prop: fgh}
\begin{split}
h_{n}\in&(0, h);\\
h_{n} f(X_n)=&\frac {\bar hf(X_n)}{1+\lfloor f(x_n)\rfloor +\lfloor g^2(X_n)\rfloor }
\le \bar h;\\
\sqrt{h_{n} }g(X_n)=&\frac {\sqrt{\bar h}g(X_n)}{\sqrt{1+\lfloor f(x_n)\rfloor +\lfloor g^2(X_n)\rfloor }}
\le \sqrt{\bar h}.
\end{split}
\end{equation}
\end{lemma}

It will be convenient to define the following:
\begin{definition}
Let $\{X_n\}_{n\in\mathbb{N}}$ be a solution of \eqref{eq:Euler}. Denote
\begin{equation}
\label{def:Un}
U_{n}:=f(X_n)h_{n}+g(X_n)\Delta W_{n+1},
\end{equation}
so that
\begin{equation}
\label{eq:attrueEMalpha}
|X_{n+1}|^\alpha=|X_n|^\alpha\left(1+ U_{n}\right)^\alpha, \quad X_0=\varsigma\neq 0, \quad n=0, 1, 2, \dots.
\end{equation}
\end{definition}

\begin{definition}
We define the following auxiliary functions for use in the proofs of our main results
\begin{enumerate}
\item $\bar \phi\in C^3(\mathbb{R};\mathbb{R})$ is such that, for some $M>0$ 
\begin{equation}\label{eq:barphi}
\bar \phi(u)=|u|^\alpha, \, u\in (-\infty, -0.5)\cup (0.5, \, \infty); \quad |\bar \phi'''(u)|\le M,\quad u\in\mathbb{R}.
\end{equation}
\item $\bar \psi\in C^3(\mathbb{R};\mathbb{R})$ is such that, for some $M>0$ 
\begin{equation}\label{eq:barpsi}
\bar \psi(u)=|u|^{-\alpha}, \, u\in (-\infty, -0.5)\cup (0.5, \, \infty); \quad |\bar \psi'''(u)|\le M,\quad u\in\mathbb{R}.
\end{equation}
\end{enumerate}
\end{definition}
The auxiliary functions $\bar \phi$ and $\bar\psi$ have the following properties, which may be proved straightforwardly.
\begin{lemma}
Let $\bar \phi$ be as defined in \eqref{eq:barphi}. Then
\[
\bar \phi(1)=1, \quad \bar \phi'(1)=\alpha, \quad \bar \phi''(1)=\alpha(\alpha-1),
\]
and for some $H>0$ we have
\begin{equation}
\label{prop:Hphi}
\bigl|\bar \phi(u)-|u|^\alpha \bigr|\le H, \quad \text{for all}\quad u\in \mathbb R.
\end{equation}
\end{lemma}

\begin{lemma}
\label{lem:barpsi}
Let $\bar \psi$ be as defined in \eqref{eq:barpsi}. Then
\[
\bar \psi(1)=1, \quad \bar \psi'(1)=-\alpha, \quad \bar \psi''(1)=\alpha(\alpha+1),
\]
and there exists $M^\ast>0$ such that $|\bar \psi(u)|\leq M^\ast$ for all $u\in\mathbb{R}$. 
\end{lemma}

Next, we show that the absolute expectations of the solutions of \eqref{eq:Euler} are bounded above. 
\begin{lemma}\label{lem:|xn|}
Let $\{X_n\}_{n\in\mathbb{N}}$ be a solution of \eqref{eq:Euler}. For each $n\in\mathbb{N}$ there exists a nonrandom $B_n\in (0, \infty)$ such that $$\mathbb E|X_{n+1}|\le B_n,\quad n\in\mathbb{N}.$$
\end{lemma}
\begin{proof}
We can assume without loss of generality that $h<1$.  Applying \eqref{prop: fgh} and the tower property of conditional expectation,
\begin{eqnarray}
\mathbb E|X_{n+1}|&\le &\mathbb E |X_n|+ \mathbb E|X_nf(X_n)h_{n}|+\mathbb E|X_ng(X_n)\Delta W_{n+1}|\nonumber\\
&\le & \mathbb E|X_n|+ h\mathbb E|X_n|+\mathbb E\left[\mathbb E \left[\left|X_ng(X_n)\Delta W_{n+1}\right|\bigr| \mathcal F_{t_{n-1}}\right]\right]\nonumber\\&\le&
2\mathbb E|X_n|+\mathbb E\left[| X_ng(X_n)|\mathbb E\left[|\Delta W_{n+1}|\bigr| \mathcal F_{t_{n-1}}\right]\right]\nonumber\\
&\le &2\mathbb E|X_n|+\mathbb E\left| X_ng(X_n)\sqrt{h_{n}}\right|\nonumber\\
&\le& 2\mathbb E|X_n|+\sqrt{\bar h}\mathbb E\left| X_n\right|\nonumber\\
& \le& 3\mathbb E|X_n|:=B_n,\label{def:Mn}
\end{eqnarray}
which is a finite deterministic constant for each $n\in \mathbb N$ by induction, since
\[
\mathbb E|X_0|=\mathbb E|\varsigma|<\infty.
\] 
\end{proof}

\subsection{Intermediate results for the proof of Theorem \ref{thm:ASstab}}
\begin{lemma}\label{lem:firstb}
Let $\{X_n\}_{n\in\mathbb{N}}$ be a solution of \eqref{eq:Euler} where \eqref{cond:stabbeta} holds. Suppose that $M$ is the bound in \eqref{eq:barphi}, and $\alpha \in (0, 1-\beta)$, where $\beta<1$ is the bound in \eqref{cond:stabbeta}. Then there exists $h_0$ such that for all $\bar h\le h_0$
\begin{equation}
\label{est:2}
\mathbb E\left[|1+U_{n}|^\alpha\bigl|\mathcal F_{t_{n-1}} \right]\le 1-2\alpha \lambda h_{n}g^2(X_n)\left(1+\bar h^2\frac {3M}{2\alpha}\right),\quad a.s.,
\end{equation}
where $\lambda:=-(1-\alpha-\beta)/2>0$.
\end{lemma}
\begin{proof}
Apply a Taylor expansion to $\bar \phi$ around $1$, using the conditional moment estimates in Remark \ref{rem:condMoments} and the fact that each $X_n$ and  $h_n$ are $\mathcal F_{t_{n-1}}$ measurable, there exists $\theta\in(0,U_n)$ such that
\begin{multline*}
\mathbb E\left[\bar \phi(1+U_{n})\bigl|\mathcal F_{t_{n-1}} \right]=1+\alpha f(X_n)h_{n}\\
+\frac {\alpha(\alpha-1)}2\mathbb E\left[f^2(X_n)h^2_{n}+2g(X_n)f(X_n)h_{n}\Delta W_{n}+g^2(X_n)[\Delta W_{n+1}]^2\bigl|\mathcal F_{t_{n-1}} \right]\\
+\mathbb E\left[ \frac{\bar \phi'''(\theta)U_{n}^3}{6}\biggl|\mathcal F_{t_{n-1}} \right],\quad a.s.
\end{multline*}
Again by the conditional moment estimates in Remark \ref{rem:condMoments} we have
\begin{eqnarray*}
\mathbb E\left[ \frac{\bar \phi'''(\theta)U_{n}^3}{6}\biggl|\mathcal F_{t_{n-1}} \right]&\le& \frac M6 \mathbb E\left[ |U_{n}^3|\bigl|\mathcal F_{t_{n-1}} \right]\\
&\le& \frac M2 \left|f^3(X_n)h^3_{n}+g^3(X_n)\mathbb E\left[\left[\Delta W_{n+1}\right]^3\bigl | \mathcal F_{t_{n-1}} \right]   \right|\\
&\le& \frac M2 \left[f^3(X_n)h^3_{n}+g^3(X_n) K_1 h^{3/2}_n \right],\quad a.s.
\end{eqnarray*}

Applying \eqref{prop: fgh} and the conditional moment estimates in Remark \ref{rem:condMoments}, noting that
\[
\frac {\alpha(\alpha-1)}2 f^2(X_n)h^2_{n}\le 0, 
\]
and 
\[
\mathbb E\left[ g^2(x_n)(\Delta W_{n+1})^2\bigl|\mathcal F_{t_{n-1}} \right]=g^2(x_n) \mathbb E\left[(\Delta W_{n+1})^2\bigl|\mathcal F_{t_{n-1}} \right]=h_n g^2(x_n),\quad a.s.,
\]
we arrive at  
  \begin{eqnarray}
\mathbb E\left[\bar \phi(1+U_{n})\biggl|\mathcal F_{t_{n-1}} \right]
&\le& 1+\alpha h_{n}\left[f(X_n)-\frac{1-\alpha}2 g^2(X_n)\right]\nonumber\\
&&+\frac {3M}2 \left[f^3(X_n)h^3_{n}+g^3(X_n) K_1 h^{3/2}_n \right]\nonumber
\\&\le& 1+\alpha h_{n}\left[f(X_n)-\frac{1-\alpha}2 g^2(X_n)\right]\nonumber\\
&&+\frac {3M}2 \bar h^2f(X_n)h_{n}+\frac {3MK_1}2 g^2(X_n) h_{n}\sqrt{\bar h}\nonumber\\&
=&1+\alpha h_{n}\left[f(X_n)\left(1+\bar h^2\frac {3M}{2\alpha}  \right)\right.\nonumber\\&&\left.-\left(\frac{1-\alpha}2-\sqrt{\bar h}\frac {3MK_1}{2\alpha}\right) g^2(X_n)\right]\nonumber\\&
=&1+\frac 12\alpha h_{n} g^2(X_n)\left[\frac{2f(X_n)}{g^2(X_n)}\left(1+\bar h^2\frac {3M}{2\alpha}  \right)\right.\nonumber\\
&&\left.-\left((1-\alpha)-\sqrt{\bar h}\frac {3MK_1}{\alpha}\right) \right],\quad a.s.\label{est:Ebarphi}
\end{eqnarray}

Now we estimate
\begin{eqnarray*}
\lefteqn{\mathbb E\left[\left|\bar \phi(1+U_{n})-|1+U_{n}|^\alpha\right|\bigl|\mathcal F_{t_{n-1}} \right]} \\&=&
\mathbb E\left[\left|\bar \phi(1+f(X_n)h_{n}+g(X_n)\Delta W_{n+1})\right.\right.\\
&&\left.\left.\qquad\qquad-|1+f(X_n)h_{n}+g(X_n)\Delta W_{n+1}|^\alpha\right|\bigl|\mathcal F_{t_{n-1}} \right] .
\end{eqnarray*}
We have,
\begin{equation}
\label{est:1+U}
|1+U_{n}|\le \frac 12 \iff  -\frac 12 \le 1+f(X_n)h_{n}+g(X_n)\Delta W_{n+1}\le  \frac 12,
\end{equation}
which is equivalent to 
 \begin{equation}
\label{est:DW} 
-\frac{ \frac 12+1+f(X_n)h_{n}}{g(X_n)}\le \Delta W_{n+1}\le  
 \frac{ \frac 12-1-f(X_n)h_{n}}{g(X_n)} .
 \end{equation}
Denote by $\Psi_n$ the  distribution function of $\Delta W_{n}$ conditioned upon $\mathcal F_{t_{n-1}}$.  Then, by \eqref {prop:Hphi}, we have 
\begin{eqnarray*}
\lefteqn{\mathbb E\left[\left|\bar \phi(1+U_{n})-|1+U_{n}|^\alpha\right|\bigl|\mathcal F_{t_{n-1}} \right]}\nonumber\\&=&\int\limits_{|1+U_{n}|\le \frac 12}
\left|\bar \phi(1+U_{n})-|1+U_{n}|^\alpha\right|d\Psi_{n}\le H\int\limits^{-(\frac 12+f(x_n)h_{n})/g(x_n)}_{-( \frac 32+f(x_n)h_{n})/g(x_n)}d\Psi_{n}\nonumber\\&=&H\Psi_{n}\left( \frac{ -\frac 12-f(x_n)h_{n}}{g(x_n)} \right)-H\Psi_{n}\left( -\frac{ \frac 32+f(x_n)h_{n}}{g(x_n)} \right),\quad a.s.
\end{eqnarray*}
Due to \eqref{est:DeltaWn} and the conditional moment estimates in Remark \ref{rem:condMoments}, and Lemma \ref{lem:3mom}, we have
\[
\left[\frac{ 3/2+f(X_n)h_{n}}{g(X_n)}   \right]^3\Psi_{n}\left( -\frac{ 3/2+f(X_n)h_{n}}{g(X_n)} \right)\le Kh^{3/2}_{n},
\]
so that
\[
\Psi_{n}\left( -\frac{ 3/2+f(X_n)h_{n}}{g(X_n)} \right)\le K\frac{h^{3/2}_{n}g^3(X_n)}{ \left[3/2+f(X_n)h_{n}\right]^3}\le K_2 h^{3/2}_{n}g^3(X_n).
\]
Similarly,
\[
\Psi_{n}\left( -\frac{ 1/2+f(X_n)h_{n}}{g(X_n)} \right)\le K_1\frac{h^{3/2}_{n}g^3(X_n)}{ \left[1/2+f(X_n)h_{n}\right]^3}\le K_3 h^{3/2}_{n}g^3(X_n).
\]
Again applying \eqref{prop: fgh}, we arrive at
\begin{multline}
\label{est:Delta2}
\mathbb E\left[\left|\bar \phi(1+U_{n})-|1+U_{n}|^\alpha\right|\bigl|\mathcal F_{t_{n-1}} \right]\\ \le K_4 h^{3/2}_{n+1}g^3(x_n)\le K_4 \bar h^{1/2} h_ng^2(x_n)\quad a.s.
\end{multline}
Combining \eqref {est:Ebarphi} with  \eqref{est:Delta2} we arrive at 
\begin{eqnarray}
\lefteqn{\mathbb E\left[|1+U_{n}|^\alpha\bigl|\mathcal F_{t_{n-1}} \right]}\nonumber\\&=&
1+\frac 12\alpha h_{n} g^2(X_n)\nonumber\\
&&\times\left[\frac{2f(X_n)}{g^2(X_n)}\left(1+\bar h^2\frac {3M}{2\alpha}  \right)-\left((1-\alpha)-\sqrt{\bar h}\frac {3MK_1+2K_4}{\alpha}\right) \right]\nonumber\\
&=&1+\frac1 2\alpha h_{n} g^2(X_n)\left(1+\bar h^2\frac {3M}{2\alpha}  \right)\nonumber\\
&&\times\left[\frac{2f(X_n)}{g^2(X_n)}-\frac{\left((1-\alpha)-\sqrt{\bar h}\frac {3MK_1+2K_4}{\alpha}\right)}{\left(1+\bar h^2\frac {3M}{2\alpha}  \right)} \right],\quad a.s.\label{est:1}
\end{eqnarray}

Let $\beta$ be defined as in \eqref{cond:stabbeta} and recall that $0<\alpha< 1-\beta$. We can choose $h_0>0$ such that for all $\bar h\le h_0$ 
\begin{equation*}
\label{cond:fg3}
\frac {1-\alpha-\sqrt{\bar h}\frac {3MK_1+2K_4}{\alpha}}{1+\bar h^2\frac {3M}{2\alpha}}>1-\alpha-\frac{1-\alpha-\beta}2=\frac{1-\alpha+\beta}2>0.
\end{equation*}
Then, for each $n\in \mathbb N$,
\[
\frac{2f(X_n)}{g^2(X_n)}-\frac{\left((1-\alpha)-\sqrt{\bar h}\frac {3MK_1+2K_4}{\alpha}\right)}{\left(1+\bar h^2\frac {3M}{2\alpha}  \right)}<\beta- \frac{1-\alpha+\beta}2=-\frac{1-\alpha-\beta}2,
\]
and \eqref{est:1} can be written as \eqref{est:2}, which completes the proof.
\end{proof}

\begin{lemma}
Let $\{X_n\}_{n\in\mathbb{N}}$ be a solution of \eqref{eq:Euler} where \eqref{cond:stabbeta} holds. Suppose $\alpha \in (0, 1-\beta)$, where $\beta<1$ is the bound in \eqref{cond:stabbeta}. Then  $\{\mu_n\}_{n\in\mathbb{N}}$ is an $\mathcal{F}_{t_n}$-martingale-difference, where
\begin{equation}\label{def:martdif}
\mu_n:=|X_n(1+U_{n})|^\alpha-\mathbb E\left[\left|X_n(1+U_{n})\right|^\alpha \bigl|\mathcal F_{t_{n-1}} \right] 
\end{equation}
and $\{\mathcal M_n\}_{n\in \mathbb N}$  is an $\mathcal F_{t_{n-1}}$-martingale, where
\begin{equation}
\label{def:mart}
\mathcal M_n:=\sum_{i=1}^{n+1}\mu_i, \quad n\in \mathbb N.
\end{equation}
\end{lemma}
\begin{proof}
Since $\left|X_n(1+U_{n})\right|^\alpha>0$, and by Jensen's inequality for concave functions, we have 
\begin{eqnarray*}
\mathbb E|\mu_n|&\le& \mathbb E\left|X_n(1+U_{n})|^\alpha-\mathbb E\left[\left|X_n(1+U_{n})|^\alpha\right|\bigl|\mathcal F_{t_{n-1}} \right]\right|\\&\le &\mathbb E|X_n(1+U_{n})|^\alpha
+\mathbb E\left[\mathbb E\left[\left|X_n(1+U_{n})\right|^\alpha\bigl|\mathcal F_{t_{n-1}} \right]  \right]\\&\le &
\left(\mathbb E|X_n(1+U_{n})|\right)^\alpha +\mathbb E|X_n(1+U_{n})|^\alpha\\
&\le& 2\left(\mathbb E|X_n(1+U_{n})|\right)^\alpha\\
& \le& 2B^\alpha_{n+1},\quad a.s.,
\end{eqnarray*}
where $B_n<\infty$ is as defined in \eqref{def:Mn}. Also, since $\mathbb E\left[\left|X_n(1+U_{n})\right|^\alpha\bigl|\mathcal F_{t_{n-1}} \right]$ is $\mathcal F_{t_{n-1}}$-measurable, 
\begin{eqnarray*}
\mathbb E\left[\mu_{n+1}\bigl|\mathcal F_{t_{n-1}} \right]&=&\mathbb E\left[\left|X_n(1+U_{n})\right|^\alpha\bigl|\mathcal F_{t_{n-1}} \right]\\
&&-\mathbb E\left[\mathbb E\left[\left|X_n(1+U_{n})\right|^\alpha\bigl|\mathcal F_{t_{n-1}} \right]\bigl|\mathcal F_{t_{n-1}} \right]\\&=&\mathbb E\left[\left|X_n(1+U_{n})\right|^\alpha\bigl|\mathcal F_{t_{n-1}} \right]-\mathbb E\left[\left|X_n(1+U_{n})\right|^\alpha\bigl|\mathcal F_{t_{n-1}} \right]\\
&=&0,\quad a.s.,
\end{eqnarray*}
from which the statement of the lemma follows.
\end{proof}

\subsection{Intermediate results for the proof of Theorem \ref{thm:ASinstab}}
\begin{lemma}\label{lem:Ualbound}
Let $\{X_n\}_{n\in\mathbb{N}}$ be a solution of \eqref{eq:Euler} where \eqref{cond:instabgamma} holds, and suppose that $\alpha\in(0,\gamma-1)$, where $\gamma>1$ is the bound in \eqref{cond:instabgamma}. Then, for some constants $K_1>0$ and $K_7>0$, a.s.,
\begin{multline}
\label{est:instab1}
\mathbb E\left[|1+U_{n}|^{-\alpha}\bigl|\mathcal F_{t_{n-1}} \right]\\\le 1-\frac 12\alpha h_{n} g^2(X_n)\left[\frac{2f(X_n)}{g^2(X_n)}-\frac{\left((1+\alpha)+\sqrt{\bar h}\frac {3MK_1+2K_7}{\alpha}\right)}{\left(1-\bar h\frac {\alpha+1}2-\bar h^2\frac {3M}{2\alpha}  \right)} \right].
\end{multline}
\end{lemma}
\begin{proof}
Let $U_n$ be  defined by \eqref{def:Un}. Proceeding as in the proof of Lemma \ref{lem:firstb}, we have
\begin{multline*}
\mathbb E\left[\bar \psi(1+U_{n})\bigl|\mathcal F_{t_{n-1}} \right]=1-\alpha f(X_n)h_{n}\\+
\frac {\alpha(\alpha+1)}2\mathbb E\left[f^2(X_n)h^2_{n}+2g(X_n)f(X_n)h_{n}\Delta W_{n}+g^2(X_n)[\Delta W_{n+1}]^2\bigl|\mathcal F_{t_{n-1}} \right]\\ +\mathbb E\left[ \frac{\bar \phi'''(\theta)U_{n}^3}{6}\biggl|\mathcal F_{t_{n-1}} \right],\quad a.s.,
\end{multline*}
leading to
  \begin{eqnarray}
\lefteqn{\mathbb E\left[\bar \psi(1+U_{n})\bigl|\mathcal F_{t_{n-1}} \right]}\nonumber\\
&\le& 1-\alpha h_{n}\left[f(X_n)\left[1- \frac{\alpha+1}2 f(X_n) h_n \right]-\frac{1+\alpha}2 g^2(X_n)\right]\nonumber\\&&+\frac {3M}2 \left[f^3(X_n)h^3_{n}+g^3(X_n) K_1 h^{3/2}_n \right]\nonumber
\\&\le& 1-\alpha h_{n}\left[f(X_n)\left(1- \frac{\alpha+1}2 f(X_n) h_n \right)-\frac{1+\alpha}2 g^2(X_n)\right]\nonumber\\&&+\frac {3M}2 \bar h^2f(X_n) h_{n}+\frac {3MK_1}2 g^2(X_n) h_{n}\sqrt{\bar h}\nonumber\\
&=&1-\alpha \bar h_{n}\left[f(X_n)\left(1- \frac{\alpha+1}2 \bar h-\bar h^2\frac {3M}{2\alpha}  \right)\right.\nonumber\\
&&\left.\qquad\qquad-\left(\frac{1+\alpha}2+\sqrt{\bar h}\frac {3MK_1}{2\alpha}\right) g^2(X_n)\right]\nonumber\\
&=&1-\frac 12\alpha h_{n} g^2(X_n)\left[\frac{2f(x_n)}{g^2(X_n)}\left(1- \frac{\alpha+1}2 \bar h-\bar h^2\frac {3M}{2\alpha}  \right)\right.\nonumber\\
&&\left.\qquad\qquad-\left((1+\alpha)+\sqrt{\bar h}\frac {3MK_1}{\alpha}\right) \right],\quad a.s.\label{est:Ebarpsi}
\end{eqnarray}

Next we estimate $\mathbb E\left[\left|\bar \psi(1+U_{n})-|1+U_{n}|^{-\alpha}\right|\bigl|\mathcal F_{t_{n-1}} \right]$.
As before, denote by $\Psi_n$ the  distribution function of $\Delta W_{n}$ conditioned upon $\mathcal F_{t_{n-1}}$.  Then
\begin{eqnarray*}
\lefteqn{\mathbb E\left[\left|\bar \psi(1+U_{n})-|1+U_{n}|^{-\alpha}\right|\bigl|\mathcal F_{t_{n-1}} \right]}\\
&=&\int\limits_{|1+U_{n}|\le \frac 12}\left|\bar \psi(1+U_{n})-|1+U_{n}|^{-\alpha}\right|d\Psi_{n}\\ &\le&  \int\limits_{|1+U_{n}|\le \frac 12}|\bar \psi(1+U_{n})|d\Psi_{n}+ \int\limits_{|1+U_{n}|\le \frac 12}|1+U_{n}|^{-\alpha}d\Psi_{n}.
\end{eqnarray*}
By Lemma \ref{lem:barpsi}, and reusing the relations \eqref{est:1+U} and \eqref{est:DW}, we may redo the equivalent calculations from the proof of Lemma \ref{lem:firstb} to show that there exists $K_5>0$ such that 
\begin{eqnarray*}
\lefteqn{\int\limits_{|1+U_{n}|\le \frac 12}|\bar \psi(1+U_{n})|d\Psi_{n}}\\&\le&  M^\ast \int\limits_{|1+U_{n}|\le 1/2}d\Psi_{n}=M^\ast\int\limits^{-\frac{ 1/2+f(X_n)h_{n}}{g(X_n)}}_{-\frac{ 3/2+f(X_n)h_{n}}{g(X_n)}}d\Psi_{n}\\&=&M^\ast\Psi_{n}\left( \frac{ -1/2-f(X_n)h_{n}}{g(X_n)} \right)-M^\ast\Psi_{n}\left( -\frac{ 3/2+f(X_n)h_{n}}{g(X_n)} \right)\\
&\le& K_5 h^{3/2}_{n}g^3(X_n).
\end{eqnarray*}
Making the substitution $u:=1+f(X_n)h_{n}+g(X_n)z$ and applying \eqref{est:DeltaWn}, we get
\begin{multline*}
 \int\limits_{|1+U_{n}|\le \frac 12}|1+U_{n}|^{-\alpha}d\Psi_{n}\\=\int\limits_{| 1+f(x_n)h_{n}+g(x_n)z |\le \frac 12}|  1+f(x_n)h_{n}+g(x_n)z|^{-\alpha} \frac 1{\sqrt{2\pi h_n}}\exp\left(-\frac{z^2}{2h_n}\right)dz\\
= \frac 1{\sqrt{2\pi h_n}}\int\limits_{| u |\le \frac 12}| u|^{-\alpha}\exp\left(-\frac{\left( \frac{u-1-f(x_n)h_{n}}{g(x_n)} \right)^2}{2h_n}\right)dz.
\end{multline*}
For $|u|\le \frac 12$, we have
\[
-\frac 32-f(X_n)h_{n} \le u-1-f(X_n)h_{n}\le -\frac 12-f(X_n)h_{n},
\]
which implies
\[
\left| u-1-f(X_n)h_{n}\right|\ge \frac 12+f(X_n)h_{n}\ge \frac 12.
\]
So, for $|u|\le \frac 12$,
\[
\exp\left(-\frac{\left( \frac{u-1-f(X_n)h_{n}}{g(X_n)} \right)^2}{2h_n}\right)\le \exp\left(-\frac 1{8h_n g^2(X_n)}\right).
\]
Applying the estimate $e^{-u}\le u^{-2}$ for $u>0$, we get
\[
\exp\left(-\frac{\left( \frac{u-1-f(X_n)h_{n}}{g(X_n)} \right)^2}{2h_n}\right)\le \exp\left(-\frac 1{8h_n g^2(X_n)}\right) \le 64h^2_n g^4(X_n).
\]
Noting that
\[
\int\limits_{| u |\le \frac 12}| u|^{-\alpha}du=\frac 2{(1-\alpha)2^{1-\alpha}},
\]
we get, for some constant $K_6>0$, 
\[
\int\limits_{|1+U_{n}|\le \frac 12}|1+U_{n}|^{-\alpha}d\Psi_{n}\le K_6 h^2_n g^4(X_n),
\]
and, therefore, applying  \eqref{prop: fgh}, we arrive at 
\[
\int\limits_{|1+U_{n}|\le \frac 12}|1+U_{n}|^{-\alpha}d\Psi_{n}\le K_6 h_{n}^{3/2} g^3(X_n).
\]
It immediately follows that 
\begin{equation}
\label{est:Delta3}
\mathbb E\left[\left|\bar \psi(1+U_{n})-|1+U_{n}|^{-\alpha}\right|\bigl|\mathcal F_{t_{n-1}} \right]\le K_7 \sqrt{\bar h}h_{n}g^2(X_n).
\end{equation}
Combining \eqref {est:Ebarpsi} with  \eqref{est:Delta3} we arrive at 
\begin{eqnarray*}
\lefteqn{\mathbb E\left[|1+U_{n}|^{-\alpha}\bigl|\mathcal F_{t_{n-1}} \right]} \\&=&1-\frac 12\alpha h_{n} g^2(X_n)\left[\frac{2f(X_n)}{g^2(X_n)}\left(1- \bar h\frac{\alpha+1}2 -\bar h^2\frac {3M}{2\alpha}  \right)\right.\\
&&\left.-\left((1+\alpha)+\sqrt{\bar h}\frac {3MK_1+2K_7}{\alpha}\right) \right]\\&=&
1-\frac 12\alpha h_{n} g^2(X_n)\left[\frac{2f(X_n)}{g^2(X_n)}-\frac{\left((1+\alpha)+\sqrt{\bar h}\frac {3MK_1+2K_7}{\alpha}\right)}{\left(1-\bar h\frac {\alpha+1}2-\bar h^2\frac {3M}{2\alpha}  \right)} \right],\quad a.s.,
\end{eqnarray*}
which completes the proof.
\end{proof}

\begin{lemma}
\label{lem:instconstrmart}
Let $\{X_n\}_{n\in\mathbb{N}}$ be a solution of \eqref{eq:Euler}.
Then
\[
M_n=\prod_{k=0}^n\frac{\left(1+ U_{k}\right)^{-\alpha}} {\mathbb E\left[|1+U_{k}|^{-\alpha}\bigl|\mathcal F_{t_{k-1}} \right]}
\] is a non-negative $\mathcal{F}_{t_n}$-martingale.
\end{lemma}
\begin{proof}
Let 
\[
Y_k:=\frac{\left(1+ U_{k}\right)^{-\alpha}} {\mathbb E\left[|1+U_{k}|^{-\alpha}\bigl|\mathcal F_{t_{k-1}} \right]},\quad k\in\mathbb{N}.
\]
Note that each $Y_k$ is nonnegative, and
\[
\mathbb E\left[ Y_k\bigl|\mathcal F_{t_{k-1}} \right]=\mathbb E\left[\frac{\left(1+ U_{k}\right)^{-\alpha}} {\mathbb E\left[|1+U_{k}|^{-\alpha}\bigl|\mathcal F_{t_{k-1}} \right]} \biggl| \mathcal F_{t_{k-1}}\right]=1, \quad a.s.,
\]
and
\[
\mathbb E\left[ Y_k \right]=\mathbb E\left[ \mathbb E\left[Y_k\bigl|\mathcal F_{t_{k-1}} \right]\right]=\mathbb E\left(1\right)<\infty.
\]
Therefore each $M_n$ is non-negative and, by  Lemma \ref{lem:prodMart}, the sequence $\{M_n\}_{n\in\mathbb{N}}$ is an $\mathcal{F}_{t_n}$-martingale.
\end{proof}

\subsection{Proofs of main results}
\begin{proof}[Proof of Theorem \ref{thm:ASstab}]
Denote, for each $n\in \mathbb N$,
\[
z_n:=|X_n|^{\alpha}, \quad v_n:=  \frac 12\alpha \lambda h_{n}g^2(X_n)\left(1+\bar h^2\frac {3M}{2\alpha}  \right).
\]
Note that $z_n\ge 0$ and $v_n\ge 0$, for each $n\in \mathbb N$, and Eq. \eqref{eq:attrueEMalpha}, \eqref{est:2}, \eqref{def:martdif}, and \eqref{def:mart} yield 
\begin{equation*}
\label{ineq:1}
z_{n+1}\le z_n-v_n+\mathcal M_n, \quad n\in\mathbb N. 
\end{equation*}
Applying \eqref{prop: fgh} and Lemma \ref{lem:|xn|}, we obtain for each $n\in \mathbb N$ that
\[
\mathbb E z_n<\infty, \quad \mathbb E v_n\le 2\alpha \lambda \left(1+\bar h^2\frac {3M}{2\alpha}  \right)<\infty.
\]
From Lemma \ref{lem:nonegdif} we conclude that both $z_n$ and $\sum_{i=0}^nv_i$ converge a.s.  Therefore $|x_n|$ converges a.s. to some nonnegative  random variable $a_0$. This implies in particular that for a.a. $\omega\in\Omega$, there exists $N_0(\omega)\in \mathbb N $  such that, for $n\ge N_0(\omega) $, 
\begin{equation}
\label{ineq:xna1}
\frac 12 a_0(\omega)\le |X_n(\omega)|\le \frac 32 a_0(\omega).
\end{equation}
We want to prove that $a_0=0$ a.s. Suppose it is not true, then there exists $\Omega_1\subset \Omega$, $\mathbb P [\Omega_1]>0$, such that, on $\Omega_1$, $a_0(\omega)>0$.
Eq. \eqref {ineq:xna1}  and the continuity of $f$ and $g$ imply that there exist positive random values $H_1$ and $H_2$ such that, for a.a. $\omega\in\Omega$,
\[
f\left(X_n(\omega)\right)\le H_1(\omega), \quad  g^2\left(X_n (\omega)\right)\le H_2(\omega), \quad \text{for}\quad n\ge N_0(\omega),
\]
and therefore,  
\[
 h_n (\omega) =\frac {\bar h}{1+\lfloor f(X_n(\omega) )\rfloor +\lfloor g^2(X_n(\omega) )\rfloor } \ge \frac {\bar h}{1+H_1(\omega) +H_2(\omega)}.
\]
By Assumption \ref{as:fg} there also exist $N_1(\omega)\in \mathbb N $  and $H_3(\omega)\in \mathbb R$,  such that for $\omega\in\Omega_1$ we have
\[
g^2(X_n(\omega))\ge H_3(\omega)\quad\text{for}\quad n\ge N_1(\omega).
\]
We have
\[
\sum_{i=0}^\infty h_i=\sum_{i=0}^{N_0} h_i+\sum_{i=N_0+1}^\infty h_i\ge \sum_{i=N_0+1}^\infty \frac { \bar h}{1+H_1+H_2}=\infty,\quad a.s.
\]
On $\Omega_1$, we have
\begin{multline*}
\sum_{i=0}^\infty v_i\ge \sum_{i=N_1}^\infty \frac 12\alpha \lambda h_{i}g^2(X_i)\left(1+\bar h^2\frac {3M}{2\alpha}  \right)\\
\ge \frac 12\alpha \lambda H_3 \left(1+\bar h^2\frac {3M}{2\alpha}  \right)\sum_{i=N_1}^\infty  h_{i}=\infty.
\end{multline*}
The contradiction thus obtained proves the result.
\end{proof}

\begin{proof}[Proof of Theorem \ref{thm:ASinstab}]
Iterating equation \eqref{eq:Euler} back to the initial value gives
\begin{equation}
\label{eq:attrueEM-alpha1}
|X_{n+1}|^{\alpha}=|\varsigma|^{\alpha}\prod_{k=0}^n\left(1+ U_{k}\right)^{\alpha}, \quad n\in\mathbb{N},
\end{equation}
which can be written
\begin{equation}
\label{eq:-alpha1}
\begin{split}
|X_{n+1}|^{\alpha}=&|\varsigma|^{\alpha} 
\cdot\left(\prod_{k=0}^n\frac{\left(1+ U_{k}\right)^{-\alpha}} {\mathbb E\left[|1+U_{k}|^{-\alpha}|\mathcal F_{t_{k-1}} \right]}\right)^{-1}
\cdot\frac 1{\prod_{k=0}^n\mathbb E\left[|1+U_{k}|^{-\alpha}|\mathcal F_{t_{k-1}} \right]}\\&=|\varsigma|^{\alpha}\cdot \frac 1{M_n}\cdot\frac 1{\prod_{k=0}^n\mathbf E\left[|1+U_{k}|^{-\alpha}|\mathcal F_{t_{k-1}} \right]},\quad n\in\mathbb{N}.
\end{split}
\end{equation}
Since, by Lemma \ref{lem:instconstrmart}, $M_n$ is a non-negative $\mathcal{F}_{t_n}$-martingale, it converges a.s. to a finite limit by Lemma \ref{lem:nnM}. Therefore $1/M_n$ is ultimately bounded away from 0 a.s. 

We proceed by contradiction, and assume that $\mathbb P\{\Omega_2 \}>0$, where \\
$\Omega_2:=\{\omega: \lim_{n\to \infty}X_n(\omega)=0 \}$.  Let $\gamma>1$ be as defined in \eqref{cond:instabgamma} in the statement of the theorem and, as required in the statement of Lemma \ref{lem:Ualbound}, we choose $\alpha$ so that $0<\alpha<\gamma-1$.  Again, by \eqref{cond:instabgamma}, we can choose $\delta\in (0, \,\gamma-1-\alpha)$
and $N(\omega, \delta)$ such that, for $\omega \in \Omega_2$ 
\[
\frac{2f(X_n)}{g^2(X_n)}>\gamma-\alpha -\delta>1, \, \text{ for all } n\ge N(\omega, \delta).
\]

Developing from estimate \eqref{est:instab1} in the statement of Lemma \ref{lem:Ualbound}, we now choose $h_0>0$ such that for all $\bar h\le h_0$ 
\begin{equation*}
\label{cond:fg31}
\frac {1+\alpha+\sqrt{\bar h}\frac {3MK_1+2K_7}{\alpha}}{1- \bar h\frac{\alpha+1}2 -\bar h^2\frac {3M}{2\alpha} }<1+\alpha+\frac{\gamma -1-\alpha-\delta}2=\frac{1+\alpha+\gamma-\delta}2.
\end{equation*}
Then, for each $n\ge N(\omega, \delta)$,
\begin{multline*}
\frac{2f(X_n)}{g^2(X_n)}-\frac{\left((1+\alpha)+\sqrt{\bar h}\frac {3MK_1+2K_7}{\alpha}\right)}{\left(1-\bar h \frac {\alpha+1}2-\bar h^2\frac {3M}{2\alpha}  \right)}\\
>\gamma -\alpha-\delta- \frac{1+\alpha+\gamma-\delta}2=\frac{\gamma-1-\alpha-\delta}2=:\bar \lambda>0.
\end{multline*}
Applying this to \eqref{est:instab1} yields that,  
for all $n\ge N(\omega, \delta)$,
\begin{equation}
\label{est:3}
\mathbb E\left[|1+U_{n}|^{-\alpha}\bigl|\mathcal F_{t_{n-1}} \right]\le 1-\frac 2\alpha \bar \lambda h_{n} g^2(X_n)
\le 1,\quad a.s.
\end{equation}
Applying this inequality to representation \eqref {eq:-alpha1} we obtain, for $n\ge N(\omega, \delta)$
\[
|X_{n+1}|^{\alpha}\ge |\varsigma|^{\alpha}\cdot \frac 1{M_n}\cdot\frac 1{\prod_{k=0}^{n\ge N(\omega, \delta)+1 }\mathbf E\left[|1+U_{k}|^{-\alpha}|\mathcal F_{t_{k-1}} \right]}.
\]
The only $n$ dependent  factor on the RHS is $M_n$ which tends to a nonzero limit, by Lemma \ref{lem:nnM}. All other factors are nonzero, so $X_n$ cannot converge to zero on $\Omega_2$. The contradiction thus obtained completes the proof.
\end{proof}

\begin{proof}[Proof of Theorem \ref{thm:pos}]
Define
\begin{equation*}
\label{def:RNpos}
\mathcal R_i=\mathcal R_i(N)=\{X_i>0, X_{i-1}>0, \dots, X_1>0, X_0>0  \}, \quad i=0, 1, \dots, N.
\end{equation*}
Since the initial value $\varsigma>0$, we have $\mathbb P[\mathcal R_0]=1$ and therefore
\begin{equation}
\label{calcul: Pn}
\begin{split}
\mathbb P\left[\mathcal R_N\right]= \mathbb P\left[\bigcap_{i=0}^N \mathcal R_i \right]&=\prod_{i=1}^N \mathbb P\left[\mathcal R_i\bigl| \mathcal R_{i-1}, \dots \mathcal R_0\right]
\\
&=\prod_{i=0}^{N-1} \mathbb P\left[X_{i+1}>0\bigl| \mathcal R_{i}\right]\\
&=\prod_{i=0}^{N-1} \mathbb P\left[X_{i+1}>0\bigl| X_{i}>0, \dots, X_0>0  \right].
\end{split}
\end{equation}
Note that, on the event $\mathcal R_i$, the inequality $X_{i+1}>0$ is equivalent to
\[
1+ h_if(X_i)+g(X_i)\Delta W_{i+1}>0,
\]
which, in turn, is equivalent to
\[
\frac{\Delta W_{i+1}}{\sqrt{h_{i}}}>-\frac{1+ h_{i}f(X_i)}{\sqrt{h_{i}}g(X_i)}.
\]
Since $f(x)$  in nonnegative for all $x$, and $\sqrt{h_{i}}g(x_i)\le \sqrt{h}$,  by \eqref{prop: fgh} we have, for $X_i>0$, 
\begin{equation}\label{eq:invBound}
-\frac{1+ h_{i}f(X_i)}{\sqrt{ h_i}g(X_i)}\le -\frac{1}{\sqrt{h_i}g(X_i)}\le -\frac{1}{\sqrt{\bar h}}.
\end{equation}
Therefore
\[
\left\{ u \ge -\frac{1}{\sqrt{\bar h}}\right\}\subseteq  \left\{ u \ge -\frac{1}{\sqrt{h_{i}}g(X_i)}\right\}\subseteq \left\{ u \ge -\frac{1+ h_{i}f(X_i)}{\sqrt{h_{i}}g(X_i)}\right\}.
\]
Recall from Remark \ref{rem:condMoments} that the random variable $\zeta_{i+1}:=\Delta W_{i+1}/\sqrt{h_{i}}$ is distributed conditionally upon $\mathcal{F}_{t_i}$ like a standard normal variable $\mathcal Z\sim\mathcal N (0, 1)$. So we may apply \eqref{calcul: Pn}, the inequality \eqref{eq:invBound}, and the fact that $\Phi$ is a symmetric distribution function to get
\begin{eqnarray*}
\mathbb P[\mathcal P_N]&=&\prod_{i=0}^{N-1} \mathbb P\left[X_{i+1}>0\bigl| \mathcal{F}_{t_i}\cap \mathcal R_{i}\}\right]\\
&=&
\prod_{i=0}^{N-1} \mathbb P\left[\frac{\Delta W_{i+1}}{\sqrt{h_{i}}}>-\frac{1+ h_{i}f(X_i)}{\sqrt{h_{i}}g(X_i)} \biggl|  \mathcal{F}_{t_i}\cap\mathcal R_{i}  \right]\\
&\ge &\prod_{i=0}^{N-1} \mathbb P\left[\zeta_{i+1}>-\frac 1{\sqrt{\bar h}} \biggl| \mathcal{F}_{t_i}\cap\mathcal R_{i}\right]\nonumber\\
&=&\prod_{i=0}^{N-1} \mathbb P\left[\mathcal  Z>-\frac 1{\sqrt{\bar h}} \right] =  \left(\mathbb P\left[\mathcal  Z>-\frac 1{\sqrt{\bar h}} \right]\right)^N\\
&=&\left(1-\Phi \left( -\frac 1{\sqrt{\bar h}} \right)  \right)^{N}=\left(\Phi \left( \frac 1{\sqrt{\bar h}} \right)\right)^N.\nonumber
\end{eqnarray*}
Fix $\varepsilon\in(0,1)$, then for all $\bar h\in (0, \bar h(\varepsilon))$ we have 
\begin{equation*}
\label{prob:varep}
\left(\Phi \left( \frac 1{\sqrt{\bar h}} \right)\right)^N\ge 1-\varepsilon,
\end{equation*}
where (since $\Phi$ is non-decreasing on $\mathbb{R}$ and therefore invertible)
\begin{equation*}
\label{def:hvarep}
\bar h(\varepsilon):= \frac 1{\left(\Phi^{-1}\left[ (1-\varepsilon)^{\frac1N}  \right]\right)^2},
\end{equation*}
providing the statement of the Theorem.
\end{proof}

\begin{remark}
\label{rem:pos}
We can derive an alternative, computationally more convenient bound on $\bar h$ using the following set of inequalities due to Sasvari \& Chen~\cite{SasChen}:
\begin{equation}\label{eq:sasChen}
\sqrt{1-e^{-x^2/2}}<\frac{1}{\sqrt{2\pi}}\int_{-x}^{x}e^{-s^2/2}ds<\sqrt{1-e^{-2x^2/\pi}}.
\end{equation}
Applying the left-hand-part of \eqref{eq:sasChen} gives
\[
\left[\Phi\left(\frac{1}{\sqrt{\bar h}}\right)\right]^N>\left[\frac{1}{2}+\frac{1}{2}\sqrt{1-e^{-1/(2{\bar h}^2)}}\right]\geq 1-\varepsilon,
\]
which can be used similarly to derive the bound
\[
h_1(\varepsilon):=\left(2\ln\left(\frac{1}{1-(2(1-\varepsilon)^{1/N}-1)^2}\right)\right)^{-1/2}.
\]
This bound is real-valued and positive for all $\varepsilon\in(0,1)$, and has the advantage of not requiring that we invert $\Phi$.
\end{remark}

\end{document}